\documentclass[12pt]{amsart}
\usepackage[latin1]{inputenc}
\usepackage{color}
\usepackage{bbm}
\usepackage{amsmath,amssymb}
\usepackage{enumerate}
\usepackage{mathrsfs}
\usepackage{verbatim}

\newtheorem{thm}{Theorem}[section]
\newtheorem{cor}[thm]{Corollary}
\newtheorem{lem}[thm]{Lemma}
\newtheorem{prop}[thm]{Proposition}
\theoremstyle{definition}
\newtheorem{defin}[thm]{Definition}

\newtheorem{example}[thm]{Example}

\theoremstyle{remark}
\newtheorem*{remark}{Remark}

\numberwithin{equation}{section}

\makeatletter
\newcommand{\lin}{\operatorname{span}}
\newcommand{\supp}{\operatorname{supp}}
\newcommand{\dist}{\operatorname{dist}}

\newcommand{\dif}{\,\mathrm{d}}

\newcommand{\charfun}{\ensuremath{\mathbbm 1}}

\DeclareMathOperator{\card}{card}
\DeclareMathOperator{\conv}{re}

\newcommand{\BS}{F}

\allowdisplaybreaks

\begin{document}
\title{Multivariate orthogonal spline systems}

\author[M. Passenbrunner]{Markus Passenbrunner}
\address{Institute of Analysis, Johannes Kepler University Linz, Austria, 4040 Linz, Alten\-berger Strasse 69}
\email{markus.passenbrunner@jku.at}


\begin{abstract}
	In this article we consider orthonormal systems consisting
	of tensor products of splines.
We show some convergence results of the corresponding orthogonal series including a.e. convergence and
unconditional convergence in $L^p$ for $1<p<\infty$, where the latter is proved
under some geometric conditions on the involved partitions that depend on the
spline order.
\end{abstract}

\maketitle 

\allowdisplaybreaks
\section{Introduction}
In this article we prove convergence results of orthogonal
series of certain tensor products of splines in the spirit of
the known results for martingales.
We begin by discussing the
situation for martingales and, subsequently, for univariate splines. 
For martingales, we use \cite{Neveu1975} and \cite{Pisier2016}  as
references.
Let $(\Omega, (\mathscr F_n), \mathbb P)$ be a filtered probability space. A
sequence of integrable functions $(X_n)_{n\geq 1}$ is a \emph{martingale} if $\mathbb
E(X_{n+1} | \mathscr F_n) = X_n$ for every $n$, where we denote by $\mathbb
E(\cdot | \mathscr F_n)$ the conditional expectation operator with respect to
the $\sigma$-algebra $\mathscr F_n$. This operator is the orthoprojector onto
the space of $\mathscr F_n$-measurable $L^2$-functions and it 
can be extended to  $L^1$. 
Observe that if $X\in L^1$, the sequence  $(\mathbb E(X|\mathscr F_n))$
is a martingale.
In this case, we have that $\mathbb E(X|\mathscr F_n)$ converges
almost surely to $\mathbb E(X|\mathscr F)$ with $\mathscr F = \sigma(
\cup_n\mathscr F_n)$.
For general scalar-valued martingales, we have the following convergence
theorem: 
 any martingale $(X_n)$ that is bounded in $L^1$ has
an almost sure limit function contained in $L^1$. 
Additionally we know that martingale differences $dX_n = X_n - X_{n-1}$ converge unconditionally in
$L^p$ for $1<p<\infty$, i.e., we have the inequality
\begin{equation}\label{eq:strong_p_mart}
	\Big\| \sum_n \varepsilon_n dX_n \Big\|_p \leq C_p \Big\| \sum_n dX_n
	\Big\|_p
\end{equation}
for all sequences of signs $(\varepsilon_n)$ and some
constant $C_p$ depending only on $p$.
More precisely, we have the following inequality of weak type:
\begin{equation}\label{eq:wt_mart}
	\sup_{\lambda >0} \lambda \cdot \mathbb P\Big( \sup_{n} \Big| \sum_{\ell\leq n}
	\varepsilon_\ell dX_\ell\Big| > \lambda\Big) \leq C \sup_{n} \| X_n \|_1,
\end{equation}
where $C$ some absolute constant.
Equation \eqref{eq:strong_p_mart} is a consequence of \eqref{eq:wt_mart} since by 
orthogonality of martingale differences $dX_n$ we have
\eqref{eq:strong_p_mart} for $p=2$ and the Marcinkiewicz interpolation theorem
then implies \eqref{eq:strong_p_mart} for every $p$ in the range $1<p<\infty$.

Consider now the special case where each $\sigma$-algebra $\mathscr F_n$ is generated by
a partition of a bounded interval $I\subset \mathbb R$ into finitely many 
intervals $(I_{n,i})_i$ of positive length as atoms of $\mathscr F_n$. In this case, $(\mathscr F_n)$ is
called an \emph{interval filtration} on $I$.
Then, the
characteristic functions $(\charfun_{I_{n,i}})$ of those atoms are a sharply localized
orthogonal basis of $L^2(\mathscr F_n)$ with respect to Lebesgue measure
$|\cdot|$. 
If we want to preserve the localization property of the basis functions, but 
at the same time consider spaces of functions with higher smoothness, a natural
candidate are spaces of piecewise polynomial functions of order $k$, given by
\begin{equation}\label{eq:defSk}
\begin{aligned}
	S_{k}(\mathscr F_n) = \{f : I\to \mathbb R\ |\ &\text{$f$ is $k-2$ times continuously
	differentiable and} \\
	&\qquad\text{a polynomial of order $k$ on
	each atom of $\mathscr F_n$} \},
\end{aligned}
\end{equation}
where $k$ is an arbitrary positive integer.
One reason for this is that $S_k(\mathscr F_n)$ admits a special basis, the so
called 
B-spline basis $(N_{n,i})_i$, that consists of non-negative and localized
functions $N_{n,i}$. 
Here, the term ``localized'' means that the support of each function $N_{n,i}$
consists of at most $k$ neighbouring atoms of $\mathscr F_n$.
A second reason is that if $(\mathscr F_n)$ is an increasing sequence of
interval $\sigma$-algebras, then the sequence of corresponding spline spaces $S_k(\mathscr
F_n)$ is increasing as well. 
Note that the aforementioned properties of the B-spline functions $(N_{n,i})$ imply that they
do not form an orthogonal basis of $S_k(\mathscr F_n)$ for $k\geq 2$.
For more information on spline functions,
see e.g. \cite{Schumaker2007}.
Let $P_n^{k}$ be the orthogonal projector onto
$S_{k}(\mathscr F_n)$ with respect to the
$L^2$  inner product on $I$  equipped with the Lebesgue measure.  
Since the space $S_{1}(\mathscr F_n)$ consists of
piecewise constant functions,
 $P_n^{1}$ is the conditional expectation
operator with respect to the $\sigma$-algebra $\mathscr F_n$ and the Lebesgue
measure.
In general, the operator $P_n^{k}$ can be written in terms of the B-spline basis $(N_{n,i})$ as
\[
	P_n^{k} f = \sum_i \int_I f(x) N_{n,i}(x)\dif x \cdot N_{n,i}^*,
\]
where the functions $(N_{n,i}^*)$,
contained in the spline space $S_{k}(\mathscr F_n)$, are the biorthogonal (or
dual) system to the
B-spline basis $(N_{n,i})$. 
 Due to the uniform boundedness of the B-spline functions
 $N_{n,i}$, we are able to insert functions $f$ in formula
\eqref{eq:Pn} that are contained not only in $L^2$, but in 
$L^1$, thereby extending the operators
$P_n^{k}$ to $L^1$.

 Similarly to the definition of martingales, we adopt the following
 notion introduced in \cite{Passenbrunner2020}:
let $(X_n)_{n\geq 1}$ be a sequence of functions in the space $L^1$. We call this
sequence a \emph{$k$-martingale spline sequence} (adapted to $(\mathscr F_n))$
if
\[
	P_n^{k} X_{n+1} = X_n,\qquad n\geq 1.
\]
The local nature of the B-splines and the nestedness of the spaces
$(S_k(\mathscr F_n))_n$ ultimately allow us to 
 transfer the classical martingale theorems discussed above
	to
	$k$-martingale spline sequences adapted to  \emph{arbitrary}
	interval filtrations ($\mathscr F_n$) and for any positive integer $k$, just by replacing
	conditional expectation operators with the spline projection operators
	$P_n^{k}$.

Assume that, for all $n$,
$\mathscr F_n$ arises from $\mathscr F_{n-1}$ by the subdivision of exactly
one atom of $\mathscr F_{n-1}$ into two intervals $L_n, R_n$ as atoms of
$\mathscr F_n$.
In the case $k=1$, spline differences $dX_n$ are the same as martingale
differences and then given by a constant multiple
of the generalized Haar function $|R_{n}| \charfun_{L_n} - |L_n|
\charfun_{R_n}$. Similarly, for $k>1$, there exists a system of (unlocalized)
orthogonal spline functions $(f_n)$ so that $dX_n$ is a constant multiple of $f_n$.
The following statements are true: 
	\begin{enumerate}
		\item $L^1$-bounded $k$-martingale spline sequences $(X_n)$
			converge almost everywhere to some $L^1$-function.
			\cite{PassenbrunnerShadrin2014,
			MuellerPassenbrunner2020}
		\item 
		Inequality \eqref{eq:strong_p_mart} holds for $k$-martingale
		spline sequences $(X_n)$ with a constant $C_{p,k}$
		depending only on $p$ and $k$ but \emph{not} on the interval filtration
		$(\mathscr F_n)$ (see \cite{GevorkyanKamont2004} for $k=2$ and
		\cite{Passenbrunner2014}  for general $k$).
	\item \label{it:cz} By using Calderon-Zygmund operator techniques, A. Kamont and K. Keryan
	\cite{KamontKeryan2021}
	showed under certain geometric conditions ($(k-1)$-regularity, cf.
	Definition~\ref{def:regularity} for $d=1$) on the filtration
	$(\mathscr F_n)$ that \eqref{eq:wt_mart} also holds for $k$-martingale
	spline sequences $(X_n)$.
\end{enumerate}

	 In this article we are concerned with similar results pertaining to
	 tensor product spline projections. 
	 Let $d$ be a positive integer and, for $j=1,\ldots,d$, let $(\mathscr
	 F_n^{j})$ be an interval filtration on the bounded interval $I^j\subset \mathbb R$. Filtrations $(\mathscr
	 F_n)$ of the form $\mathscr F_n = \mathscr F_n^1\otimes \cdots \otimes
	 \mathscr F_n^d$ will be called an \emph{interval filtration} on 
	 the $d$-dimensional rectangle $I^1 \times \cdots \times I^d$. 
	 Then, the atoms of $\mathscr F_n$ are of the form
	 $A_1\times \cdots\times A_d$ with atoms $A_j$ in $\mathscr F_n^j$.
	 For a tuple $k=(k_1,\ldots,k_d)$ consisting of $d$ positive integers,
	 denote by $P_n^{k}$ 
	 the orthogonal projector with respect to
	 $d$-dimensional Lebesgue measure $|\cdot|$ onto the tensor product
	 spline space $S_k(\mathscr F_n) = S_{k_1}(\mathscr F_n^{1}) \otimes \cdots \otimes
	 S_{k_d}(\mathscr F_n^{d})$.
	In \cite{Passenbrunner2021} we show that an $L^1$-bounded sequence of functions $(X_n)$
	with $P_n^k X_{n+1} = X_n$ converges almost everywhere to some
	$L^1$-function.

Now we assume that, for $n\geq 1$,  $\mathscr F_{n}$ is of the  form that $\mathscr
F_{n}=\mathscr F_{n}^{1} \otimes \cdots \otimes \mathscr F_{n}^d$ arises 
from $\mathscr F_{n-1} = \mathscr F_{n-1}^{1} \otimes \cdots \otimes \mathscr
F_{n-1}^d$
in the way that there exists a coordinate $\delta_0\in \{1,\ldots,d\}$ so that  
$\mathscr F_{n}^\delta = \mathscr F_{n-1}^\delta$ for $\delta\neq \delta_0$ and
$\mathscr F_{n}^{\delta_0}$ 
arises from $\mathscr F_{n-1}^{\delta_0}$ by splitting exactly one atom of $\mathscr
F_{n-1}^{\delta_0}$ into two atoms of $\mathscr F_{n}^{\delta_0}$.
	 In Section~\ref{sec:tensor_def} we describe an
	 orthonormal system $(f_\ell)$ consisting of tensor products of spline
	 functions so that there exists an increasing sequence of integers
	 $(M_n)$ satisfying
	 \[
		 S_k(\mathscr F_n) = \lin\{ f_\ell : \ell \leq M_n \}\qquad
		 \text{for all $n$}.
	 \]
	 In the special case $k=(1,\ldots,1)$, if $M_{n-1} < \ell \leq M_n$,
	 those functions $f_\ell$ are given by (a constant multiple of) the tensor product of
		one generalized Haar function in direction $\delta_0$ with 
	 $(d-1)$ characteristic
	 functions of atoms in $(\mathscr F_n^\delta)$ in the directions
	 $\delta\neq \delta_0$. Therefore, in this case, $(f_\ell)$ is a
	 martingale difference sequence.

	 Here, we extend the result concerning a.e. convergence from
	 \cite{Passenbrunner2021} and show that partial sums of the form
	 $X_n=\sum_{\ell\leq n} a_\ell f_\ell$, that are uniformly
	 $L^1$-bounded, converge almost
	 everywhere. Moreover, we give sufficient geometric conditions on the
	 filtration $(\mathscr F_n)$ (cf. Theorem~\ref{thm:main}) so that we
	 have the following weak type inequality,
	 similar to \eqref{eq:wt_mart}:
	 \begin{equation}\label{eq:wt_intro}
		 \sup_{\lambda > 0} \lambda \cdot \Big| \Big\{ \sup_n \Big|
			 \sum_{\ell\leq n} \varepsilon_\ell a_\ell f_\ell \Big|
		 > \lambda\Big\} \Big| \leq C \sup_n \| X_n \|_1
	 \end{equation}
	 for some constant $C$,
	 all sequences $(\varepsilon_n)$ of signs and all sequences of
	 coefficients $(a_n)$.

	 We note two things.
	 Firstly, by specializing Theorem~\ref{thm:main} to $d=1$,
	 our sufficient conditions on $(\mathscr F_n)$ for inequality
	 \eqref{eq:wt_intro} are less restrictive than the sufficient condition in
	 item \eqref{it:cz} ($(k-1)$-regularity implies $k$-regularity,
	 cf. Definition~\ref{def:regularity} and the succeeding remark).
	 Secondly, for $d\geq 2$, our sufficient conditions allow for an arbitrary
	 ratio of sidelengths of atoms of $\mathscr F_n$, meaning that the
	 rectangles that are atoms of $\mathscr F_n$ can be very long in one
	 direction and very short in another direction.

	 The organization of the article is as follows. In
	 Section~\ref{sec:prelim} we collect known results about polynomials and
	 spline functions that are used in the
	 sequel. In Section~\ref{sec:tensor_orth} we construct multivariate
	 orthonormal spline functions $(f_n)$. Section~\ref{sec:main} contains
	 the formulation of our main result (Theorem~\ref{thm:main}) that
	 inequality \eqref{eq:wt_intro} is valid under
	 certain geometric conditions on the filtration $(\mathscr F_n)$ that
	 are also defined and analyzed here. Finally,
	 Section~\ref{sec:proof_main} contains the proof of the main result.

	 \section{Preliminaries}\label{sec:prelim}
\subsection{Polynomials}
We will need the following multi-dimensional version of Remez' theorem
(see \cite{Ga2001,BrudnyiGanzburg1973}). If
	$p(x)=\sum_{\alpha\in \Lambda} a_\alpha x^\alpha$ is a $d$-variate polynomial
	where $\Lambda$ is a finite set containing $d$-dimensional multiindices, the
	degree of $p$ is defined as $\max\{\sum_{i=1}^d \alpha_i :
\alpha\in \Lambda\}$.
Recall that a convex body in $\mathbb R^d$ is a compact, convex set with non-empty interior.
\begin{thm}[Remez, Brudnyi, Ganzburg]
	Let $d\in\mathbb N$, $V\subset \mathbb R^d$ a convex body and $E\subset V$ a measurable
	subset. Then, for all polynomials $p$ of degree $r$ on $V$,
	\begin{equation*}
		\| p \|_{L^\infty(V)} \leq \bigg( 4d \frac{|V|}{|E|}\bigg)^r \| p
		\|_{L^\infty(E)}.
	\end{equation*}
\end{thm}
We have the following corollary:
\begin{cor} \label{cor:remez}
Let $p$ be a polynomial of degree $r$ on a convex body $V\subset \mathbb R^d$. Then
\begin{equation*}
	\big|\big\{ x \in V : |p(x)| \geq (8d)^{-r} \|p\|_{L^\infty(V)} \big\}\big| \geq |V|/2.
\end{equation*}
\end{cor}
\begin{proof}
	This follows from an application of the above theorem to the set $E = \{x\in
		V : |p(x)| \leq (8d)^{-r} \|p\|_{L^\infty(V)}\}$.
\end{proof}

\subsection{Spline spaces}
Consider an \emph{interval $\sigma$-algebra $\mathscr F$}, i.e. a
$\sigma$-algebra that is generated by
a partition of a bounded interval $I\subset \mathbb R$ into finitely many 
intervals of positive length as atoms of $\mathscr F$. Let $k$ be an arbitrary
positive integer.
Let $S_k(\mathscr F)$ be the spline space of order $k$ corresponding to
the $\sigma$-algebra $\mathscr F$ defined in~\eqref{eq:defSk} and let $(N_i)$ be
the B-spline basis of $S_k(\mathscr F)$ that forms a partition of unity.

In the next result and in what follows,
we use the notation $A(t)\lesssim_x B(t)$ if there exists a
constant $c$ depending only on the order parameter $k$ and on $x$ so that $A(t)\leq c B(t)$
for all $t$, where $t$ denotes all implicit or explicit dependencies that the
symbols $A$ and $B$ might have. Similarly we use the symbols $\gtrsim_x$ and $\simeq_x$.

\begin{prop}[B-spline stability]\label{prop:lpstab} 
	Let $1\leq p< \infty$ and $g=\sum_{j} a_j N_{j}$ be a linear
	combination of B-splines. Then,
\begin{equation}\label{eq:lpstab}
	|a_j|\lesssim |K_j|^{-1/p}\|g\|_{L^p(K_j)},\qquad \text{for all $j$},
\end{equation}
where $K_j\subseteq \supp N_{j}$ is an atom of $\mathscr F$  having  maximal length. Additionally,
\begin{equation}\label{eq:deboorlpstab}
	\|g\|_p\simeq \Big(\sum_{j} |a_j|^p \cdot |\supp N_j|\Big)^{1/p}.
\end{equation}
\end{prop}
The two inequalites \eqref{eq:lpstab} and \eqref{eq:deboorlpstab} are Lemma 4.1
and Lemma 4.2 in \cite[Chapter 5]{DeVoreLorentz1993}, respectively.
For more information on spline functions,
see e.g. \cite{Schumaker2007}.

Let $P$ be the orthogonal projector onto
$S_k(\mathscr F)$ with respect to the
$L^2$  inner product $\langle\cdot,\cdot\rangle$ on $I$  equipped with the Lebesgue measure.  
Since the space $S_{1}(\mathscr F)$ consists of
piecewise constant functions,
 for the choice $k=1$, the operator $P$ is the conditional expectation
operator with respect to the $\sigma$-algebra $\mathscr F$ and the Lebesgue
measure.
This orthogonal projector is uniformly bounded on $L^\infty$ by a
constant depending only on the spline order $k$, which is content of the
following celebrated theorem by A. Shadrin \cite{Shadrin2001}:

\begin{thm}\label{thm:shadrin}
	Let $P$ be the orthogonal projector onto {$S_k({\mathscr F})$}
	with respect to the canonical inner product in $L^2(I)$.
	Then, 
	\begin{equation*}
		\| P : L^\infty(I) \to L^\infty(I) \| \lesssim 1.
	\end{equation*}
\end{thm}
Since orthogonal projectors are self-adjoint, this also implies that the
operators $P$ are uniformly bounded on $L^1(I)$ by the same constant.

The operator $P$ can be written in terms of the
B-spline basis $(N_{i})$ and its dual basis $(N_i^*)$ as 
\begin{equation}\label{eq:Pn}
	P f = \sum_i \int_I f(x) N_{i}(x)\dif x \cdot N_{i}^*,
\end{equation}
where the functions $N_i^*\in S_k(\mathscr F)$ are given by the conditions
$\langle N_i^*, N_j\rangle = \delta_{ij}$ for all $i,j$ and $\delta_{ij} = 1$ if
$i=j$ and $\delta_{ij}=0$ otherwise, where we denote $\langle f,g\rangle = \int_I f(x)g(x)\dif x$.
The dual B-spline functions $N_i^*$ can be written in terms of the B-spline
basis $N_{i}^* = \sum_j a_{ij}
N_{j}$ for some coefficients $(a_{ij})$.
Those coefficients $(a_{ij})$ admit some fast decay away from the diagonal:
\begin{thm}[\cite{PassenbrunnerShadrin2014}]\label{thm:maintool}
There exists a constant $q\in(0,1)$ depending only on the spline order
$k$ so that 
\[
|a_{ij}|\lesssim \frac{q^{|i-j|}}{|\conv(\supp N_i , \supp N_j)|},\qquad
\text{for all } i,j,
\]
where $\conv(A,B)$ denotes the smallest interval containing both sets $A,B$.
\end{thm}

	\subsection{Totally positive matrices}\label{sec:tp}
We say that a matrix $B=(b_{ij})_{i,j=1}^n$ is \emph{totally positive} if for
any choice of ${\bf m_1, m_2}\subset \{1,\ldots,n\}$ 
with the same cardinality,
the determinant of the matrix $B({\bf m_1}, {\bf m_2})$, resulting from
$B$ by taking the rows with indices in ${\bf m_1}$ and columns with indices in
${\bf m_2}$, is non-negative.

For totally positive matrices, we have the following well known lemma, which can be found 
in \cite{deBoorJiaPinkus}. 
\begin{lem} \label{lem:tp0}
	If $B\in \mathbb R^{n\times n}$ is invertible and totally positive,
	then, for any integer interval $\bf m$ $\subset \{1,\ldots,n\}$, so is
	the principal submatrix $C := B({\bf m},{\bf m})$ of $B$
	and
	\[
		0 \leq (-1)^{i+j} C^{-1}(i,j) \leq (-1)^{i+j} B^{-1}(i,j),
		\qquad i,j\in \bf m.
	\]
\end{lem}

A proof of this result is contained in \cite{deBoor2012}.
We also have the following straightforward extension, whose proof is similar to
the proof of Lemma~\ref{lem:tp0}, but we include it here for
completeness.
\begin{lem}\label{lem:tp}
	If $B\in \mathbb R^{n\times n}$ is invertible and totally positive,
	then, for every ${\bf m} \subset \{1,\ldots,n\}$, so is
	the principal submatrix $C := B({\bf m},{\bf m})$ of $B$
	and
	\begin{equation}\label{eq:inverse_relative}
		0 \leq (-1)^{i+j} C^{-1}(i,j) \leq (-1)^{i+j} B^{-1}(i,j),
		\qquad i,j\in\bf m.
	\end{equation}
\end{lem}
\begin{proof}
	Let ${\bf m_1} \subseteq \{1,\ldots,n\}$ and ${\bf m} ={\bf m_1}
	\setminus \{\ell\}$ for some $\ell\in {\bf m_1}$.
	Denote by $C$ the matrix $B({\bf m}, {\bf m})$ and by $D$ the matrix
	$B({\bf m_1}, {\bf m_1})$, which are both totally positive.
	Then we show that if $D$ is invertible, so is $C$ and the matrix $E$, given by
	\begin{equation}\label{eq:E}
		E(i,j) = D^{-1}(i,j) - \frac{ D^{-1}(i,\ell)
		D^{-1}(\ell,j)}{D^{-1}(\ell,\ell)},\qquad i,j\in {\bf m},
	\end{equation}
	is the inverse of $C$. 

	First we note that by the Hadamard inequality for totally positive
	matrices (see e.g. \cite[Theorem 1.21]{Pinkus2010}), we have
	$\det D \leq D(\ell,\ell)\det C$ and we obtain by the invertibility of
	$D$ that $\det D>0$ and therefore also $\det C>0$. Now we invoke the
	formula $D^{-1}(\ell,\ell)= \det C/\det D$ to see that
	$D^{-1}(\ell,\ell)>0$ and the formula in \eqref{eq:E} makes sense.

	Next, we calculate for $r,s\in {\bf m}$
	\begin{align*}
		(CE)(r,s) &= \sum_{u\in {\bf m}} C(r,u)E(u,s) = \sum_{u\in {\bf
		m}} D(r,u) \Big( D^{-1}(u,s) - \frac{ D^{-1}(u,\ell)
		D^{-1}(\ell,s)}{D^{-1}(\ell,\ell)}\Big) \\
		&= \delta_{rs} - D(r,\ell)D^{-1}(\ell,s) - \big(\delta_{r\ell} -
		D(r,\ell)
		D^{-1}(\ell,\ell)\big) \frac{D^{-1}(\ell,s)}{D^{-1}(\ell,\ell)}.
	\end{align*}
	Since $r\neq \ell$ we have $\delta_{r\ell} = 0$, which gives that $(CE)(r,s) =
	\delta_{rs}$ for $r,s\in{\bf m}$. A similar calculation yields
	$(EC)(r,s) = \delta_{rs}$ for $r,s\in {\bf m}$ which implies that the
	matrix $E$, given by formula \eqref{eq:E}, is
	indeed the inverse of $C$.

	Since $D$ is totally positive, for any choice of $i,j\in{\bf m}$ 
	we know that $(-1)^{i+j} D^{-1}(i,j)\geq 0$ 
	and  $ (-1)^{i+j}D^{-1}(i,\ell) D^{-1}(\ell,j) /
	D^{-1}(\ell,\ell)\geq 0$. By equation \eqref{eq:E}, this implies the inequality
	\[
		0\leq (-1)^{i+j} E(i,j)  = (-1)^{i+j} C^{-1}(i,j) \leq
		(-1)^{i+j}D^{-1}(i,j),\qquad i,j\in {\bf m}.
	\]
	Therefore, by induction on the cardinality of ${\bf m}$, this implies
	inequality \eqref{eq:inverse_relative}.
\end{proof}

\subsection{Application of Lemma~\ref{lem:tp} to B-spline matrices}\label{sec:discuss}
Observe that the matrix $(a_{ij})$ -- satisfying $N_i^* = \sum_j a_{ij}N_j$ for
all $i$ -- is given by the inverse of the B-spline Gram
matrix $B=(\langle N_i, N_j\rangle)_{i,j=1}^n$.
This matrix is totally positive, which is 
a consequence of
the fact that the kernel $N_{i}(x)$, depending on the variables $i$ and $x$, 
is totally positive \cite[Theorem 4.1, Chapter 10]{Karlin1968} and
the so called basic composition formula \cite[Chapter 1, Equation
(2.5)]{Karlin1968}.
This means that we can apply Lemma~\ref{lem:tp} to
submatrices of B-spline Gram matrices.

Let $(I_{i})_i$ be an enumeration of the atoms  of $\mathscr F$ for consecutive integers $i$ in the way that
if $i<j$ then $I_{i}$ is to the left of $I_{j}$.
Let $A,B$ be two atoms of $\mathscr F$. If $A=I_{i}$ and $B=I_{j}$ for two
integers $i,j$, we set $d_{\mathscr F}(A,B) = j-i$.
Denote $\BS_i = \supp N_i$, 
and by $A(x)$ the atom of $\mathscr F$ containing the point
$x\in I$.
Denote
by $K_i\subset F_i$ an atom of $\mathscr F$ contained in $F_i$ having maximal
length.
Then, Theorem~\ref{thm:maintool} implies the following pointwise estimate for the dual B-spline
functions:
\begin{equation}\label{eq:dualspline}
	| N_i^*(x) | \lesssim \frac{q^{ |d_{\mathscr F}(K_i, A(x))| }}{
	|\conv(F_i, A(x))| },\qquad \text{$1\leq i\leq n,$  $x\in I$.}
\end{equation}
Choose an arbitrary subset ${\bf m} \subset \{1,\ldots,n\}$. Let $N_i^{{\bf m}*}$,
$i\in \bf m$ be the dual functions to $\{ N_i : i\in {\bf m} \}$.
Then, we can write
\[
	N_i^{{\bf m}*} = \sum_{j\in \bf m} a_{ij}^{\bf m} N_j,
\]
where the coefficients $(a_{ij}^{\bf m})_{i,j\in \bf m}$ are given as the inverse to the matrix $(\langle
N_i,N_j\rangle)_{i,j\in\bf m}$.
Since the matrix $B = (\langle N_i, N_j\rangle)_{i,j=1}^n$ is totally
positive and invertible, 
we can invoke Lemma \ref{lem:tp} to deduce from Theorem~\ref{thm:maintool} that 
\begin{equation}\label{eq:geom_aijm}
	|a_{ij}^{\bf m}|\lesssim \frac{ q^{|i-j|} }{|\conv (\BS_i, 
	\BS_j)|},\qquad i,j\in{\bf m}.
\end{equation}
Therefore, we have the following estimate for the functions $N_i^{ {\bf m}*}$ similar to \eqref{eq:dualspline}:
\begin{equation}\label{eq:dual_estimate}
	|N_i^{ {\bf m}*}(x)| \lesssim \frac{ q^{ |d_{\mathscr F}( K_i, A(x) )|} }{|\conv(F_i,
	A(x))|},\qquad i\in{\bf m}, x\in I. 
\end{equation}
This estimate does not depend on the subset ${\bf m}$ of $\{1,\ldots,n\}$.

If ${\bf m_1} \subset {\bf m}\subset \{1,\ldots,n\}$ with $ {\bf m} \setminus {\bf m_1} =
\{i\}$, we have 
$ N_{i}^{ {\bf m}*} \in \lin \{ N_j : j\in {\bf m} \}$ and 
\[
	\langle N_{i}^{ {\bf m}*}, N_\ell \rangle = 0,\qquad \ell\in {\bf m_1}.
\]
This means that $N_{i}^{ {\bf m}*}$ is 
orthogonal to the span of $ \{N_\ell : \ell\in \bf m_1 \}$.
We know by Lemma~\ref{lem:tp} applied to the subset $\{i\}$ of ${\bf m}$ that
\begin{equation}\label{eq:lowerestimate}
	|a_{ii}^{\bf m}| \geq \frac{1}{\langle N_i, N_i\rangle} \geq 
	\frac{1}{|\BS_i|}.
\end{equation}
 Using \eqref{eq:lowerestimate} 
and the local stability \eqref{eq:lpstab} of B-splines, we
obtain (for $1\leq p<\infty$)
\begin{equation}\label{eq:J1}
	\| N_{i}^{ {\bf m}*} \|_{L^p(K_{i})} \gtrsim |K_{i}|^{1/p} |a_{ii}^{\bf
	m}| \gtrsim
	|K_{i}|^{1/p
	- 1}.
\end{equation}
On the other hand,  by \eqref{eq:deboorlpstab} and \eqref{eq:geom_aijm},
\begin{equation}\label{eq:J2}
	\int |N_{i}^{ {\bf m}*}|^p \simeq \sum_{j\in {\bf m}} |a_{ij}^{\bf m}|^p
	|\BS_j| \lesssim
	\sum_{j\in {\bf m}} \frac{ q^{ p |i-j|} }{ |\conv(\BS_{i},
	\BS_{j})|^p} |\BS_j| \lesssim |\BS_{i}|^{1-p}\lesssim |K_{i}|^{1-p}.
\end{equation}
Inequalities \eqref{eq:J1} and \eqref{eq:J2} together
imply that $\|N_{i}^{ {\bf m} *}\|_p \simeq |K_{i}|^{1/p - 1}$.

\subsection{Orthogonal spline functions}\label{sec:franklin}
Let $(\mathscr F_n)_{n\geq 0}$ be an \emph{interval filtration} on an interval $I$,
which means that $(\mathscr F_n)$ is an increasing sequence of interval $\sigma$-algebras on
the interval~$I$.

Additionally, we assume that $(\mathscr F_n)$ satisfies $\mathscr F_0 = \{\emptyset,I\}$ and,
is in \emph{standard form},
meaning that  for all $n\geq 1$,
$\mathscr F_n$ arises from $\mathscr F_{n-1}$ by the subdivision of exactly
one atom of $\mathscr F_{n-1}$ into two intervals $L_n$ and $R_n$ as atoms of
$\mathscr F_n$.
Then, the codimension
of $S_k(\mathscr F_{n-1})$ in $S_k(\mathscr F_n)$ is one and thus there exists a
unique (up to sign) function $f_n\in S_k(\mathscr F_n)$ that is orthonormal to $S_k(\mathscr
F_{n-1})$.
In the case $k=1$, the function $f_n$ is a constant multiple of the generalized
Haar function $|R_{n}| \charfun_{L_n} - |L_n|
\charfun_{R_n}$. 

We denote $d_n(A,B) = d_{\mathscr F_n}(A,B)$ for two atoms $A,B$ of $\mathscr
F_n$ and we let $A_n(x)$ be the atom of
$\mathscr F_n$ containing the point $x\in I$.

The functions $(f_n)$ satisfy that for every $n$, there exists an atom $J_n$ of $\mathscr F_n$ with the
following properties (\cite{Passenbrunner2014}).
\begin{enumerate}
	\item $|d_n(J_n, L_n)| \leq k$.
	\item There exists a support $F$ of a B-spline
		function in $S_k(\mathscr F_n)$ with $F\supset J_n$ and
		$|F|\lesssim |J_n|$.
	\item Pointwise estimate for $f_n$:
		\begin{equation}\label{eq:estfranklin}
			|f_n(x)| \lesssim \frac{q^{| d_n(A_n(x),J_n) | }
			|J_n|^{1/2} } { |\conv(J_n, A_n(x))| },\qquad x\in I.
		\end{equation}
	\item $\| f_n \|_p \simeq |J_n|^{1/p - 1/2}$ for $1\leq p\leq \infty$.
\end{enumerate}
We say that $J_n$ is the \emph{characteristic interval} of the function
$f_n$.
Additionally we have the following lemma, which is also contained in
\cite{Passenbrunner2014}.
\begin{lem}\label{lem:Jint}
	Let $V\subset I$ be an interval. Then, the cardinality of the set
	\[
		\{ n : J_n \subset V, |J_n|\geq |V|/2\}
	\]
	is bounded by some constant depending only on $k$.
\end{lem}

\subsection{Tensor product splines}
\label{sec:tensor}
Let $d$ be a positive integer and let for any coordinate $\delta \in
\{1,\ldots,d\}$ the sequence $(\mathscr F_n^{\delta})$ be an interval filtration 
on the bounded interval $I^{\delta}$, generated by
the intervals $(I_{n,i}^{\delta})_i$. Put $I = I^1\times \cdots\times I^d$.
If $\mathscr F_n = \mathscr F_n^1\otimes \cdots\otimes \mathscr F_n^d$, then
the sequence $(\mathscr F_n)$ is called an \emph{interval filtration} on the
$d$-dimensional rectangle $I$.
Every $\sigma$-algebra $\mathscr F_n$ is generated by the finite, mutually
disjoint family 
$\{ I_{n,i} : i\in \Lambda \}$, $\Lambda\subset \mathbb Z^d$, of 
$d$-dimensional rectangles given by $I_{n,i} = \prod_{\ell=1}^d
I_{n,i_\delta}^\delta$ for $i\in \Lambda$.
We assume that
$\Lambda$ is of the form $\Lambda^1\times \cdots\times \Lambda^d$ where for each
$\ell=1,\ldots,d$, $\Lambda^\ell$ is a finite set of consecutive integers and the
rectangles $I_{n,i}$ have the property that they are ordered in the same way as
$\mathbb R^d$, i.e., if $i,j\in\Lambda$ with $i_\ell < j_\ell$ then the
projection of $I_{n,i}$ onto the $\ell$th coordinate axis lies to the left of the
projection of $I_{n,j}$ onto the $\ell$th coordinate axis.
For $x\in I$, let $A_n(x)$ be the uniquely determined atom (rectangle)  $A\in\mathscr
F_n$ so that $x\in A$. For two atoms $A,B\in \mathscr F_n$, define
$d_n(A,B) := d_{\mathscr F_n}(A,B):= j-i\in\mathbb Z^d$ if $A=I_{n,i}$ and $B=I_{n,j}$. For
$s\in\mathbb Z^d$, we put $|s|_1 = \sum_{j=1}^d |s_j|$.

For each $\ell=1,\ldots,d$, let $k_\ell$ be a positive integer.
Define the tensor product spline space of order $k= (k_1,\ldots,k_d)$ associated to $\mathscr
F_n$ as
\[
	S_k(\mathscr F_n) := S_{k_1}(\mathscr F_n^{1}) \otimes \cdots \otimes
	S_{k_d}(\mathscr F_n^{d}). 
\]
The space $S_k(\mathscr F_n)$ admits the tensor product B-spline basis
$(N_{n,i})_{i}$ defined by
\[
	N_{n,i} = N_{n,i_1}^1\otimes \cdots \otimes N_{n,i_d}^d,
\]
where $(N_{n,i_\ell}^\ell)_{i_\ell}$ denotes the B-spline basis of
$S_{k_\ell}(\mathscr F_n^{\ell})$ that forms a partition of unity.
The support $F_{n,i}$  of
$N_{n,i}$ is composed of at most $k_1\cdots k_d$ neighouring atoms of $\mathscr
F_n$. 
Consider the orthogonal projection operator $P_n$ onto $S_k(\mathscr F_n)$ with respect to
the $d$-dimensional Lebesgue measure.
A direct consequence of Shadrin's theorem \ref{thm:shadrin} and using its tensor
product structure is that $P_n$ is uniformly bounded on $L^\infty(I)$ (and
therefore also on $L^1(I)$).
Using the B-spline basis and its biorthogonal system $(N_{n,i}^*)$, 
the orthogonal projector $P_n$ is given by
\begin{equation}\label{eq:rep_Pn}
	P_n f =\sum_i \int_I f(x)N_{n,i}(x)\dif x \cdot  N_{n,i}^{*},\qquad
	f\in L^1(I).
\end{equation}

In the following, the symbols  $\lesssim,\gtrsim,\simeq$  are used with
the same meaning as before, but note that the dependence of the constants on the
parameter $k=(k_1,\ldots,k_d)$ also includes an implicit dependence on the
dimension $d$.

The dual B-spline functions $N_{n,i}^*$ admit the following crucial geometric decay estimate 
\begin{equation}	
	\label{eq:mainestimate}
	| N_{n,i}^*(x) | \lesssim \frac{ q^{|d_n(K_{n,i}, A_n(x))|_1} }{ |\conv
		(F_{n,i},
	A_n(x))|}, \qquad x\in I
\end{equation}
for some constant $q\in [0,1)$ that depends only on  $k$, where $\conv(A,B)$
denotes the smallest, axis-parallel rectangle containing both sets  $A,B$ and
$K_{n,i}\subset F_{n,i}$ is an atom of $\mathscr F_n$ having maximal volume.
This inequality is a consequence of Theorem~\ref{thm:maintool} and
the fact that $N_{n,i}^*$ is the
tensor product of one-dimensional dual B-spline functions.
Inserting this estimate in formula \eqref{eq:rep_Pn} for $P_n f$ and as $F_{n,i}$ 
consists of
at most $k_1\cdots k_d$ neighbouring atoms of $\mathscr F_n$, setting $C_k :=
C(k_1\cdots k_d) q^{-|k|_1}$, we get the
pointwise estimate
\begin{equation}
	\label{eq:estPn}
	|P_n f(x)| \lesssim \sum_{A \text{ atom of }\mathscr F_n} 
		\frac{q^{|d_n(A,A_n(x))|_1}}{|\conv(A , A_n(x))|}\int_A
		|f(t)|\dif t,\qquad
		f\in L^1(I).
\end{equation}
Introducing the maximal function
\begin{equation}\label{eq:max_fct}
	\mathscr M f(x) = \sup_{n} \sum_{A\text{ atom of }\mathscr F_n}
	\frac{\rho^{|d_n(A,A_n(x))|_1}}{|\conv(A , A_n(x))|}\int_A
		|f(t)|\dif t,
	\qquad x\in I, f\in L^1(I)
\end{equation}
for some fixed parameter $\rho\in[0,1)$,
	we have the following Theorem \cite{Passenbrunner2021}. 
\begin{thm}\label{prop:maximal}
	The maximal function $\mathscr M$ is of weak type
	(1,1), i.e. there exists a constant $C$ depending only on the dimension
	$d$ and on the parameter $\rho<1$, so that we
	have the inequality
	\[
		| \{ \mathscr M f > \lambda \}| \leq  \frac{C}{\lambda}\|
			f\|_{L^1},\qquad \lambda>0,\ f\in L^1(I).
	\]
\end{thm}

This theorem and estimate \eqref{eq:estPn} imply that for any $f\in L^1$, the
sequence of orthogonal projections $P_n f$ on the spline spaces $S_k(\mathscr
F_n)$ of the function $f$ converges almost everywhere with respect to
$d$-dimensional Lebesgue measure  (see also \cite{Passenbrunner2021}).

\section{Orthonormal tensor spline functions}\label{sec:tensor_orth}
	\label{sec:tensor_def}
In this section, we construct orthonormal spline functions based on a given
interval filtration $(\mathscr F_n)_{n\geq 0}$ on a $d$-dimensional rectangle
$I$ and based
on the parameter $k = (k_1,\ldots,k_d)$ containing the orders $k_\delta$ of the
polynomials in direction $\delta$.

For $n\geq
1$ assume that $\mathscr F_{n}$ is of the  form that $\mathscr
F_{n}=\mathscr F_{n}^{1} \otimes \cdots \otimes \mathscr F_{n}^d$ arises 
from $\mathscr F_{n-1} = \mathscr F_{n-1}^{1} \otimes \cdots \otimes \mathscr
F_{n-1}^d$
in the way that there exists a coordinate $\delta_0\in \{1,\ldots,d\}$ so that  
$\mathscr F_{n}^\delta = \mathscr F_{n-1}^\delta$ for $\delta\neq \delta_0$ and
$\mathscr F_{n}^{\delta_0}$ 
arises from $\mathscr F_{n-1}^{\delta_0}$ by splitting exactly one atom of $\mathscr
F_{n-1}^{\delta_0}$ into two atoms of $\mathscr F_{n}^{\delta_0}$.
If the sequence $(\mathscr F_n)$ of interval $\sigma$-algebras on a
$d$-dimensional rectangle $I$ has these properties, we say that $(\mathscr F_n)$ is
of \emph{standard form}.
Additionally, assume that $\mathscr F_0=\{\emptyset,I\}$.

Let $(f_{0,m})_{m=1}^{M_0}$ be an orthonormal basis of the space of polynomials $S_k(\mathscr
F_0)$ on $I$ that are of order $k_\delta$ in direction $1\leq \delta\leq d$.
Fix the index $n\geq 1$.
In the following construction of orthonormal spline functions
$(f_{n,m})_{m=1}^{M_n}$
 corresponding to the $\sigma$-algebra $\mathscr F_{n}$, we assume 
without restriction that $\delta_0=1$. For other values of $\delta_0$, the
construction proceeds similarly with obvious modifications.
Let $f$ be the $L^2$-normalized function that is 
contained in the space $S_{k_1}(\mathscr F_{n}^1)$ and 
orthogonal to $S_{k_1}(\mathscr F_{n-1}^1)$ and let $J$ be its corresponding
characteristic interval (cf. Section \ref{sec:franklin}).
For $j\geq 2$, let $(N_{n,i}^j)_{i\in\Lambda_n^j}$ be the B-spline basis of $S_{k_j}(\mathscr
F_n^j)$. 
We successively define index sets $\Omega_m^j\subset \Lambda_n^j$, functions $f_{n,m} = f\otimes
D_m^{2}\otimes\cdots\otimes D_m^{d}$ and corresponding characteristic intervals
$J_{n,m}^j$ for $j\geq 2$ and $1\leq m\leq \prod_{j=2}^d
|\Lambda_n^j|=:M_n$, which will be given by the following
inductive procedure on $m$.
Let $\pi$ be an arbitrary permutation of the set
$\{2,\ldots,d\}$, which is allowed to depend on the value of $n$.

If $m=1$, and for $j\geq 2$, we choose $\nu_j \in \Lambda_n^j$ arbitrarily and
set $\Omega_1^j = \{ \nu_j\}$.
Define
\[
	D_1^j = \frac{N_{n,\nu_j}^j }{\| N_{n,\nu_j}^j \|_2 },\qquad j\geq 2,
\]
and let $J_{n,1}^j$ be an atom of $\mathscr F_n^j$ contained in the
support of $N_{n,\nu_j}^j$ having maximal length.

Assume that $m\in \{2,\ldots, M_n\}$ and that $\Omega_\ell^j, D_\ell^j,
J_{n,\ell}^j$ are
defined for $\ell < m$ and $j\geq 2$.
Let $j_0=\max\{ \pi(j) : \Omega_{m-1}^j \subsetneq \Lambda_n^j \}$.
Then we distinguish three cases for the parameter $j$ according to the value of $\pi(j)$. 

\begin{enumerate}
	\item $\pi(j)= j_0$: 
		Choose $\mu_{j} \in \Lambda_n^{j}\setminus \Omega_{m-1}^{j}$ arbitrarily and
		set $\Omega_{m}^{j} = \Omega_{m-1}^{j} \cup \{ \mu_{j} \}$. Let
		$D_{m}^{j}$ be the $L^2$-normalized function contained in $\lin \{
			N_{n,\ell}^{j} :
			\ell\in \Omega_{m}^{j} \}$ that is orthogonal to $\lin \{
				N_{n,\ell}^{j} : \ell\in\Omega_{m-1}^{j}\}$. This function
			is uniquely given up to sign.
			Let $J_{n,m}^j$ be an  atom of $\mathscr F_n^j$
			contained in the support of $N_{n,\mu_j}^j$ with maximal
			length.
	\item $\pi(j)>j_0$: we choose $\mu_j \in \Lambda_n^j$
		arbitrarily, set 
		$\Omega_{m}^j = \{\mu_j\}$ and $D_{m}^j =  N_{n,\mu_j}^j /
		\| N_{n,\mu_j}^j \|_2$.
		Let $J_{n,m}^j$ be a largest atom of $\mathscr F_{n}^j$
		contained in the support of $N_{n,\mu_j}^j$.
	\item $\pi(j)<j_0$: we put $\Omega_{m}^j = \Omega_{m-1}^j$, $D_{m}^j =
		D_{m-1}^j$ and $J_{n,m}^j = J_{n,m-1}^j$.
\end{enumerate}
Then, for any $m\in \{1,\ldots,M_n\}$,  we define the \emph{characteristic
interval} $J_{n,m}$ of $f_{n,m}=f\otimes
D_m^{2}\otimes\cdots\otimes D_m^{d}$ by
\[
	J_{n,m} = J \times J_{n,m}^{2} \cdots \times J_{n,m}^d,
\]
which is an atom of $\mathscr F_n$. By construction, each atom of $\mathscr F_n$
appears at most $k_2\cdots k_d$ times among the sets $J_{n,m}$ for $1\leq m\leq
M_n$.

By the discussion in Section \ref{sec:discuss}, we know that for any $m$ and
$\pi(j)=j_0$, the function $D_{m}^j$ is a renormalization of the
function $N_{\nu_j}^{{\bf m}*}$ with ${\bf m} = \Omega_{m}^j$ and $\nu_j$ being the
only element in the set $\Omega_{m}^j \setminus \Omega_{m-1}^j$ (with the
understanding that $\Omega_0^j = \emptyset$).
Combining the estimates \eqref{eq:dual_estimate}, \eqref{eq:J1}, \eqref{eq:J2},
and \eqref{eq:estfranklin}, we obtain the following pointwise estimate for the
functions $f_{n,m}$:
\begin{equation}\label{eq:pw}
	| f_{n,m}(x) | \lesssim \frac{ q^{|d_n(A_n(x), J_{n,m})|_1}
	|J_{n,m}|^{1/2}}{|\conv(J_{n,m}, A_n(x)) |}, \qquad x\in I.
\end{equation}
The same estimates imply $\|f_{n,m}\|_p \simeq |J_{n,m}|^{1/p - 1/2}$ for $1\leq
p\leq\infty$.

By construction, the functions $(f_{n,m})$ are orthonormal in $L^2(I)$.
For arbitrary $n\geq 0$ and $0\leq m\leq M_n$, we denote by $P_{n,m}$ the orthogonal projection operator onto the span of 
\[
	S_k(\mathscr F_{n-1}) \cup \{ f_{n,\mu} : 1\leq \mu\leq m \},
\]
which we decompose as 
\begin{equation}\label{eq:Pnm}
		P_{n,m} f(x) = P_{n,0}f(x) + \sum_{\mu = 1}^m \langle f,
		f_{n,\mu}\rangle f_{n,\mu}.
	\end{equation}
	We now show that $P_{n,m}$ is uniformly bounded on $L^1$ by a constant
	depending only on $k$ (and therefore also on $d$). Indeed, the
	operator $P_{n,0}$ equals the operator $P_{n-1}$ in
	Section~\ref{sec:tensor}, for which we already know the uniform
	$L^1$-boundedness. This means, in order to estimate $\|P_{n,m} : L^1\to
	L^1\|$, we only have to estimate
\[
	\Big\| \sum_{\mu\leq m} \langle f,f_{n,\mu}\rangle f_{n,\mu}\Big\|_1
	\lesssim \sum_{\text{$A$ atom of $\mathscr F_n$}} \Big(\sum_{\mu\leq m}
	\frac{ |J_{n,\mu}| q^{|d_n(J_{n,\mu},A)|_1}} {|\conv(J_{n,\mu}, A)|}\Big)
	\int_A |f(t)|\dif t,
\]
which follows from \eqref{eq:pw}.
Since each atom of $\mathscr F_n$ occurs at most $k_1\cdots k_d$ times among the
sets $J_{n,\mu}$, $1\leq \mu\leq m$, the latter sum over $\mu$ is bounded by a
constant depending only on $k$, which already gives the claimed
uniform boundedness of $\|P_{n,m} : L^1 \to L^1\|$.

Estimate \eqref{eq:pw} also yields the following pointwise bound for $P_{n,m}f$
by the maximal function $\mathscr Mf$ introduced in \eqref{eq:max_fct}.

\begin{prop}\label{prop:proj_max}
	For any $n\geq 0$ and any $m\in \{1,\ldots,M_n\}$, we have the
	inequality 
\[
	|P_{n,m}f(x)| \lesssim \mathscr Mf(x),\qquad x\in I,
\]
for $\rho = q^{1/2}$ (with $q$ as in \eqref{eq:pw})  in the definition \eqref{eq:max_fct} of $\mathscr
M$.
\end{prop}
\begin{proof}
	Since we already know the desired bound for $P_{n,0}f(x)$ by
	\eqref{eq:estPn}, it suffices
	to consider the second term in equation \eqref{eq:Pnm}. Using estimate \eqref{eq:pw}, we obtain
	\begin{equation}
		\label{eq:mf}
	\begin{aligned}
		\Big|\sum_{\mu\leq m} \langle f,f_{n,\mu}\rangle
		f_{n,\mu}(x)\Big| \lesssim
		\sum_{\text{ $A$ atom of $\mathscr F_n$ }}
		\Big(\sum_{\mu\leq m}\frac{q^{|d_n(J_{n,\mu}, A_n(x))|_1 +
		|d_n(J_{n,\mu},A)|_1
	}|J_{n,\mu}|}{|\conv(J_{n,\mu},A) |\cdot |\conv(J_{n,\mu}, A_n(x)) |} \Big)
		\int_A |f(t)|\dif t.
	\end{aligned}
	\end{equation}
	Define $\rho = q^{1/2}$. For any atom $A$ of $\mathscr F_n$ and any
	index $\mu\leq m$ we
	have the inequalities $|d_n(A_n(x), A)|_1 \leq |d_n(J_{n,\mu}, A_n(x))|_1 +
	|d_n(J_{n,\mu},A)|_1$. Moreover, since for any coordinate $\delta$, we
	have $|\conv(A_n^\delta(x), A^\delta)| \leq |\conv(J_{n,\mu}^\delta,
	A_n^{\delta}(x))| + |\conv(J_{n,\mu}^\delta, A^\delta)|$, we also have
	the inequality
	\begin{align*}
		\frac{|J_{n,\mu}|}{|\conv(J_{n,\mu},A) |\cdot |\conv(J_{n,\mu},
		A_n(x)) |} 
		\leq 2^d\frac{1}{|\conv(A_n(x), A)|}.
	\end{align*}
	Inserting this in \eqref{eq:mf}, we obtain
	\begin{align*}
		\Big|\sum_{\mu\leq m} \langle f,f_{n,\mu}\rangle
	f_{n,\mu}(x)\Big| &\lesssim \sum_{\text{$A$ atom of $\mathscr F_n$}}
		\frac{\rho^{|d_n(A_n(x), A)|_1}}{|\conv(A_n(x), A)|} \Big(
		\sum_{\mu\leq m} \rho^{|d_n (J_{n,\mu}, A)|_1}\Big)
		\int_A |f(t)|\dif t \\
		&\lesssim \mathscr Mf(x),
	\end{align*}
	since each atom of $\mathscr F_n$ only occurs at most $k_1\cdots k_d$
	times among the sets $J_{n,\mu}$ for $1\leq \mu\leq m$.
\end{proof}
Combining this pointwise inequality for $P_{n,m}f$ and the weak type estimate
for the maximal function $\mathscr M f$ contained in Theorem \ref{prop:maximal}
yields -- as in \cite{Passenbrunner2021} --  that for every $f\in L^1(I)$, the series $\sum_{\ell} \langle
f,f_\ell\rangle f_\ell$
converges almost everywhere with respect to $d$-dimensional Lebesgue measure, if we use the
rearrangement $(f_\ell)$ of the functions $(f_{n,m})$ described in the following.

\subsection{Rearrangement of the functions $f_{n,m}$}\label{sec:rearr}
To each function $f_{n,m}$ we associate the $\sigma$-algebra $\mathscr F_{n}$
(so that $J_{n,m}$ is an atom of $\mathscr F_n$).
Now we enumerate the functions $(f_{n,m})$ as $(f_\ell)_{\ell\geq 0}$ according
to the lexicographic ordering on the pairs $(n,m)$. If $f_{n,m} = f_{\ell}$ for
some indices $n,m,\ell$, then we define the associated $\sigma$-algebra
$\mathscr A_\ell$ to the function $f_\ell$ by $\mathscr A_\ell = \mathscr
F_{n}$ and also we define the characteristic interval $J_\ell = J_{n,m}$
corresponding to the function $f_\ell$.
Observe that -- as opposed to the situation for $(\mathscr F_n)$ in standard
form -- two different values $\ell_1,\ell_2$ can give 
$\mathscr A_{\ell_1} = \mathscr A_{\ell_2}$ by this definition.

\subsection{(Quasi-)Dyadic extension of interval
$\sigma$-algebras}\label{sec:dyadic}
Let $\mathscr F$ be an interval $\sigma$-algebra  on a one-dimensional interval and let $(A_j)_{j=1}^m$ be an
enumeration of the intervals that are atoms of $\mathscr F$. For each $j=1,\ldots,m$, let $A_j =
L_j\cup R_j$ be a decomposition of $A_j$ into two disjoint intervals.
Define
\begin{equation}\label{eq:dyadic}
	\mathscr D_{1,\ell}(\mathscr F) = \sigma(\mathscr F, \{ L_j, R_j :
	1\leq j\leq \ell\}),\qquad \ell=1,\ldots,m 
\end{equation}
to be the $\sigma$-algebra generated by $\mathscr F$ and by the splitting of the
first $\ell$ atoms $A_j$ into the two intervals $L_j, R_j$.
For an integer $\nu\geq 1$, we define
inductively
\[
	\mathscr D_{\nu+1,\ell}(\mathscr F) = \mathscr D_{1,\ell} \big( \mathscr
	D_{\nu, 2^{\nu-1}m}(\mathscr F)\big),\qquad \ell = 1,\ldots,2^\nu m.
\]

Let $\mathscr A = \mathscr A^{1}\otimes \cdots \otimes \mathscr A^{d}$ be an
interval $\sigma$-algebra on a $d$-dimensional rectangle. 
A \emph{quasi-dyadic extension}  of $\mathscr A$ consists of an
interval filtration  $(\mathscr A_n)_{n\geq 0}$ in standard form with $\mathscr
A_0 = \mathscr A$ so that each $\mathscr
A_n$ equals 
\[
	\mathscr D_{\nu_1,\ell_1} (\mathscr A^{1}) \otimes \cdots \otimes
	\mathscr D_{\nu_d,\ell_d} (\mathscr A^{d})
\]
for some $(\nu_i,\ell_i)$
satisfying $|\nu_i - \nu_j| \leq 1$ for all $i,j=1,\ldots,d$. 

If the intervals $L_j$ and $R_j$ in \eqref{eq:dyadic} are chosen so that $|L_j|
= |R_j| = |A_j|/2$, a quasi-dyadic extension $(\mathscr A_n)$ of $\mathscr A$
will be called a \emph{dyadic extension} of $\mathscr A$.

\section{Unconditional convergence of multivariate orthogonal spline series}\label{sec:main}

Let $\mathscr F$ be an interval $\sigma$-algebra on a bounded interval $U$. Let
$(N_i)$ be the B-spline basis of $S_r(\mathscr F)$ for some positive integer
$r$. We say that a subset $A\subset U$ is a \emph{B-spline support of order $r$
in $\mathscr F$}, if $A$ is the support of one of the B-spline
functions $N_i$. Observe that $A$ is a union of at most $r$ neighbouring atoms of $\mathscr
F$.

\begin{defin}(Regularity)\label{def:regularity}
Let $(\mathscr F_n)$ be an interval filtration on $I$ with
$\mathscr F_n = \mathscr F_n^1 \otimes \cdots \otimes \mathscr F_n^d$ and let
$r=(r_1,\ldots,r_d)$ be a $d$-tuple of positive integers. 
\begin{enumerate}

\item
We say that $(\mathscr F_n)$ is
\emph{$r$-regular
with parameter $\gamma$}, if for all $\delta\in \{1,\ldots,d\}$, for all $n$ and
for any two B-spline supports $A,B$ of order $r_\delta$ in $\mathscr F_n^\delta$
with vanishing Euclidean distance, we have
\[
	\max \Big(\frac{|A|}{|B|}, \frac{|B|}{|A|} \Big) \leq \gamma.
\]

\item
We say that $(\mathscr F_n)$
is \emph{direction
$r$-regular with parameter $\beta$} if, 
for all $\delta\in
\{1,\ldots,d\}$,
 for all strictly decreasing sequences $(A_j)_{j=1}^\beta$ of atoms in 
 some $\mathscr F_{n_j}$, respectively,
and for all B-spline supports  
$B$ of order $r_\delta$ in $\mathscr F_{n_1}^{\delta}$ with
$A_{1}^{\delta}\subset B$, the set $B$ is not a B-spline support of order
$r_\delta$ in $\mathscr F_{n_\beta}^{\delta}$.
\end{enumerate}
\end{defin}

\begin{remark}We make a few comments on the definition of regularity and
	direction regularity.
\begin{enumerate}
	\item If $d=1$, for any interval filtration $(\mathscr
		F_n)$ and any choice of positive integer $r$, the filtration
		$(\mathscr F_n)$ is direction $r$-regular with parameter $r+1$.
	\item It is easily seen that if $(\mathscr F_n)$ is $r$-regular with
		parameter $\gamma$, then, for every $m$ with $m_i\geq r_i$ for
		every $i\in \{1,\ldots,d\}$, the filtration $(\mathscr F_n)$ is  
		$m$-regular for some parameter $\gamma'$. The same statement
		holds for direction regularity instead of regularity.
	\item Observe that regularity and direction regularity do not 
		assume any condition on the relative sidelengths of atoms of
		$\mathscr F_n$, we only have regularity  in every fixed 
	direction and direction regularity which basically says that we are not
	allowed to refine too often while neglecting a particular direction.
	\item If $(\mathscr F_n)$ is a quasidyadic interval filtration (meaning
		that $(\mathscr F_n)$ is a quasidyadic extension of
		the trivial $\sigma$-algebra on a $d$-dimensional rectangle), the sequence $(\mathscr
		F_n)$ is direction $(1,\ldots,1)$-regular for some parameter
		$\beta$.

		This observation also allows us to give examples of interval
		filtrations that are direction $(1,\ldots,1)$-regular, but not $r$-regular
		for any choice of $d$-tuples of integers~$r$.
	\item The notions of regularity and direction regularity are invariant
		under the rearrangement $(\mathscr A_n)$ of $(\mathscr F_n)$ described
		in Section~\ref{sec:rearr}.
\end{enumerate}
\end{remark}

Given an interval filtration $(\mathscr F_n)_{n\geq 0}$ in standard form on a $d$-dimensional rectangle $I$ with $\mathscr F_0 = \{\emptyset, I\}$
and a parameter
$k=(k_1,\ldots,k_d)$, we let $(f_n)$ be the sequence of orthonormal spline
functions constructed in Section~\ref{sec:tensor_def} in the order described in
Section~\ref{sec:rearr}.

\begin{thm}\label{thm:main}
Let $k=(k_1,\ldots,k_d)$ be a tuple of positive integers. 
Let $(\mathscr F_n)$ be an interval filtration in standard form on a $d$-dimensional
rectangle $I$ (with $\mathscr F_0 = \{\emptyset, I\}$) that is $k$-regular with parameter $\gamma$ and direction $k$-regular with parameter
	$\beta$. 

	Then, for all $f\in L^1$ and all signs $(\varepsilon_n)$,
	\[
		\Big| \Big\{ \sup_M \big| \sum_{n\leq M} \varepsilon_n \langle
		f,f_n\rangle f_n \big| >
	\lambda \Big\}\Big| \lesssim_{\gamma,\beta} \frac{\|f\|_1}{\lambda},\qquad \lambda>0.
	\]
\end{thm}

As a corollary we obtain, using the Marcinkiewicz interpolation theorem, that
the orthonormal system $(f_n)$ is an unconditional basic sequence in $L^p$ for $1<p<\infty$ under
the conditions on $(\mathscr F_n)$ stated in Theorem~\ref{thm:main}.

\subsection{Analysis of direction regularity}

\begin{lem}\label{lem:dir}Let $(\mathscr F_n)$ be an interval filtration
	on a $d$-dimensional rectangle $I$. Let $(\mathscr F_n)$ be  $k$-regular 
	 with parameter $\gamma$ and 
	direction $m$-regular with parameter $\beta$.

	For any direction $i=1,\ldots,d$, if $m_i > k_i$ then
	$(\mathscr F_n)$ is direction $m'$-regular for some parameter $\beta'$
	depending only on $k_i,\gamma$ and $\beta$ where $m' = m - e_i$ and
	$e_i=(0,\ldots,0,1,0,\ldots,0)$ is the canonical unit vector in
	direction $i$.
\end{lem}
\begin{proof}
	Let $i\in \{1,\ldots,d\}$ and $m$, $m'$ be as in the assumptions. We want to prove direction
$m'$-regularity of $(\mathscr F_n)$ for some parameter $\beta'$. Fix $\delta\in
\{1,\ldots,d\}$. If $\delta\neq i$, the condition for direction $m'$-regularity follows
directly from direction $m$-regularity. Thus, assume $\delta=i$ and assume that $(\mathscr
F_n)$ is not direction $m'$-regular with some parameter $C$. This means that there
exists an increasing sequence $n_1 < \cdots < n_C$ of indices and a strictly
decreasing sequence of sets $(A_j)_{j=1}^C$ so that $A_j$ is an atom in
$\mathscr F_{n_j}$ for all $j=1,\ldots,C$ and there exists a B-spline support $B\supset
A_1^\delta$ in $\mathscr F_{n_1}^{\delta}$ of order $m_\delta-1$ that is still  a B-spline
support of order $m_\delta-1$ in $\mathscr F_{n_C}^\delta$.
Define $\mathscr G_j = \mathscr F_{n_j}^\delta$ for $j=1,\ldots,C$.

Let $B_1$ be a B-spline support of order $m_\delta$ in $\mathscr G_1$
with $B_1 \supset B\supset A_1^\delta$. Since $(\mathscr F_n)$ is direction
$m$-regular with parameter $\beta$, we know that $B_1$ is not a B-spline support
of order $m_\delta$ in $\mathscr G_\beta$.
Let $B_2\supset B\supset A_1^\delta$ be a B-spline support of order $m_\delta$
in $\mathscr G_{\beta}$ with $B_2 \subset B_1$. Then we know that $B_2$ is not
a B-spline support of order $m_\delta$ in $\mathscr G_{2\beta}$.
Inductively, we get a strictly decreasing sequence $(B_\ell)$ so that $B_{\ell}$
is a B-spline support of order $m_\delta$ in $\mathscr G_{(\ell-1)\beta}$, but not
in $\mathscr G_{\ell\beta}$ and $B_\ell \supset
B\supset A_1^\delta$ for all $\ell$.
	Defining $D_\ell = B_\ell \setminus B$, we know that $D_{\ell}\setminus
	D_{\ell+k_\delta}$ contains
	a B-spline support of order $k_\delta$ in the $\sigma$-algebra $\mathscr
	G_{(\ell+k_\delta)\beta}$.
	By $k_\delta$-regularity of $(\mathscr G_j)$ in direction $\delta$,
	\begin{equation}\label{eq:D}
		|B|\leq |B_{\ell + k_\delta}|\leq \gamma | D_\ell \setminus
		D_{\ell+k_\delta} | = \gamma |D_\ell| - \gamma
		|D_{\ell+k_\delta}|\leq \gamma|D_\ell|,
	\end{equation}
	Using the estimate $|D_{\ell+k_\delta}|\leq |B_{\ell+k_\delta}|$, this also gives 
	\begin{equation}\label{eq:geom}
		|D_{\ell+k_\delta}|\leq \frac{\gamma}{1+\gamma}|D_\ell|.
	\end{equation}
	By $k_\delta$-regularity (since $B$ contains a
	B-spline support of order $k_\delta\leq m_\delta-1$ in every $\mathscr
	G_j$) we have
	$|D_1|\leq \gamma |B|$ and $|B|\leq \gamma |D_\ell|$ by \eqref{eq:D}.
	Those
	inequalities, together with the geometric decay \eqref{eq:geom}, are only possible
	if $\ell$ is bounded in terms of $\gamma$ and $k_\delta$, which implies
	an upper bound of $C$ in terms of $\gamma, k_\delta$ and $\beta$.
\end{proof}

By induction, this implies that if a $k$-regular filtration $(\mathscr F_n)$ is
direction $m$-regular with $m_i \geq k_i$ for all directions
$i\in\{1,\ldots,d\}$, we obtain that $(\mathscr F_n)$ is also direction
$k$-regular with different constants.

\begin{example}\label{ex:reg}
	On the other hand, we now discuss the possibility of a filtration $(\mathscr
F_n)$ that is $k$-regular and direction $k$-regular, but not direction
$m$-regular with $m_i\leq k_i$ for all $i\in\{1,\ldots,d\}$ and $m_\delta <
k_\delta$ for at least one direction $\delta$.

Fix $m_\delta=k_\delta-1$ and let $\varepsilon>0$ arbitrary. 
If the $\sigma$-algebra $\mathscr F(\varepsilon)$ on the interval $[-1,1)$
	is generated by the intervals $[-1, - \varepsilon),
	[\varepsilon, 1)$ and the $m_\delta$ intervals $[- \varepsilon +
		2j\varepsilon/m_\delta,  - \varepsilon +
		2(j+1)\varepsilon/m_\delta)$ for $j=0,\ldots,m_\delta-1$, the $k_\delta$-regularity
		parameter of the $\sigma$-algebra $\mathscr F(\varepsilon)$ is smaller than $2$
		and we can refine the two intervals $[-1,-\varepsilon)$ and
			$[\varepsilon,1)$ in $\mathscr F(\varepsilon)$ a number
				of at least $|\log \varepsilon|$ times without increasing
				the bound $2$ for the $k_\delta$-regularity
				parameter.
					
				This implies that we can give examples of
				interval filtrations $(\mathscr F_n)$ (that are
				$k$-regular and direction $k$-regular) on
				$d$-dimensional rectangles that are not
				$m$-regular and not
				$m$-direction regular by 
				using the $\sigma$-algebras $\mathscr F(1/\ell)$
				and its refinements described above for all
				positive integers $\ell$ in the construction of
				the filtration $(\mathscr F_n^\delta)$.
			\end{example}
			Example~\ref{ex:reg} and Lemma~\ref{lem:dir} explain
				the choice of the same order $k$ for
				regularity  and direction regularity in the
				formulation of Theorem~\ref{thm:main}. 

				If $(\mathscr F_n)$ is an interval filtration on
				a $d$-dimensional rectangle in standard form
				with $\mathscr F_0$ being the trivial
				$\sigma$-algebra, then we denote by $(\mathscr
				A_n)$ and $(J_n)$ its rearrangement and the
				sequence of characteristic intervals described
				in Section~\ref{sec:rearr}, respectively.

\begin{lem}\label{lem:comb} 
	Let $(\mathscr F_n)$ be an interval filtration on a $d$-dimensional
	rectangle $I$ in standard form (with $\mathscr F_0 = \{\emptyset, I\}$)
	that is $k$-regular 
	with parameter $\gamma$ and direction
	$k$-regular with parameter $\beta$.
	Let $(C_{n})_{n\in\Lambda}$ be a decreasing sequence of sets  with
	the following  properties.
	\begin{enumerate}
		\item $C_{n}$ is an atom of $\mathscr A_{n}$ for all
			$n\in\Lambda$.
		\item There exists $s\in\mathbb Z^d$ so that $d_{n}(C_{n},
			A_{n}) := d_{\mathscr
				A_{n}}(C_{n},
			J_{n}) = s$ for all $n\in\Lambda$.
		\item There exists a direction $\delta\in \{1,\ldots,d\}$ so
			that
			\begin{enumerate}
				\item $J_{n}^{\delta} \subseteq
					J_{m}^{\delta}$ for all $n,m\in\Lambda$
					with $n\geq m$,
				\item denoting $n_1=\min\Lambda$, there are at least $k_\delta$ atoms of 
					$\mathscr A_{n_1}^\delta$ between the
					sets $C_{n_1}^\delta$ and
					$J_{n_1}^\delta$.
			\end{enumerate}
	\end{enumerate}

 	Then, the cardinality $\card\Lambda$ of $\Lambda$ admits the bound
	\[
		\card\Lambda \lesssim_{\gamma,\beta}  \sum_{j\neq \delta}(1+ |s_j|).
	\]
\end{lem}
\begin{proof}
	Assume without restriction that the sequence of
	$\sigma$-algebras $(\mathscr A_{n})_{n\in\Lambda}$ is strictly increasing.
	This can be done, since we know that $(C_{n})$ is decreasing, which
	means that if $n_i < \cdots <n_{i+m}$ with
	$n_i,\ldots,n_{i+m}\in\Lambda$ and $\mathscr A_{n_i} = \cdots
	=\mathscr A_{n_{i+m}}$ we know that $C_{n_{i+m}} = \cdots =C_{n_i}$. But
	also $d_{n_{i+\ell}}(J_{n_{i+\ell}}, C_{n_{i+\ell}}) = s$ is constant
	for all $\ell =1,\ldots,m$ which implies that $J_{n_{i+m}} = \cdots =
	J_{n_i}$.  By construction of the intervals $J_n$ in
	Section~\ref{sec:tensor_def}, we get that $m\leq
k_1\cdots k_d$.

Using the notation $d_{n}^\delta = d_{ \mathscr A_{n}^\delta}$, we obtain (by
(2) and (3b))
\[
	|s_\delta| = |d_{n}^\delta( C_{n}^\delta, J_{n}^\delta)| \geq
	k_\delta+1,\qquad
	n\in \Lambda,
\]
which means that there exists a set $\Delta$ between $C_{n_1}^\delta$ and
$J_{n_1}^\delta$ that is a B-spline support of order
$k_\delta$ in all $\sigma$-algebras $\mathscr A_{n}^{\delta}$ for $n\in\Lambda$
and so that the Euclidean distance between $\Delta$ and $J_{n_1}^\delta$ is zero. 
This also implies that for all $n\in\Lambda$, the Euclidean distance between
$\Delta$ and $J_{n}^\delta$ is zero.
Since $J_{n}^\delta$ is the largest
atom in $\mathscr A_n^\delta$ contained in some B-spline support of order
$k_\delta$ in $\mathscr A_{n}^\delta$, we know that, by
$k_\delta$-regularity of the
filtration $(\mathscr A_{n}^\delta)$,
\begin{equation}\label{eq:Jdelta}
	(k_\delta\gamma)^{-1} |\Delta| \leq| J_{n}^\delta	| \leq \gamma
	|\Delta|,\qquad n\in\Lambda.
\end{equation}
We split the index set $\Lambda$ into the (not mutually disjoint)
	subsets
	\[
		\Gamma_i = \{ n\in\Lambda :  \mathscr A_{n}^i \neq \mathscr
		A_{m}^i \text{ for all $m\in\Lambda, m<n$} \},\qquad
		i\in\{1,\ldots,d\}.
	\]
	Every $n\in\Lambda$ is contained in some set $\Gamma_i$ since we assumed
	that $(\mathscr A_n)_{n\in\Lambda}$ is strictly increasing.
	Additionally, $n_1 = \min\Lambda\in \Gamma_i$ for every
	$i\in\{1,\ldots,d\}$.

First we count the indices in the set $\Gamma_\delta$.
By Lemma~\ref{lem:Jint},  inequality \eqref{eq:Jdelta} is only possible a
constant number of times $c_\delta$ depending
on $k_\delta,\gamma$, which implies 
$|\Gamma_\delta| \leq c_\delta$. 

Next, let $i=1,\ldots,d$ with $i\neq \delta$ arbitrary. 
For $n,m\in \Gamma_i$ with $m<n$ we assume that either $J_{n}^i$ is a strict
subset of $J_{m}^i$ or $J_{n}^i \cap J_{m}^i = \emptyset$. This can be
done without restriction since the case $J_{n}^i = J_{m}^i$, by Lemma~\ref{lem:Jint},
can increase the cardinality of $\Gamma_i$ only by a factor depending on $k_i$.
In both cases, there exists an atom of $\mathscr A_{m}^i$ that is not an atom
of $\mathscr A_{n}^i$ and is contained in the set $\conv( J_{n_1}^i,
C_{n_1}^i )$. The number of atoms of $\mathscr
A_{n_1}^i$ contained in $\conv( J_{n_1}^i,
C_{n_1}^i )$ is $1+|s_i|$. If we now assume that $|\Gamma_i| >
2^\beta(1+|s_i|)$, then there exists a strictly decreasing sequence
$(A_{n}^i)_{n\in\Omega}$
with $\Omega \subset \Gamma_i$ and $\card\Omega\geq
\beta$ so that $A_{n}^i$ is an atom of $\mathscr
A_{n}^i$ for all $n\in \Omega$. This implies that there exists a strictly decreasing
sequence of sets $(A_{n})_{n\in\Omega}$ so that 
$A_n$ is an atom of $\mathscr A_n$ and $A_n^\delta \subset \Delta$ for all
$n\in\Omega$.
Since $\Delta$ is a B-spline support of order $k_\delta$ in the $\sigma$-algebra
$\mathscr A_{n}^\delta$ for every $n\in\Omega$,
by direction $k$-regularity of $(\mathscr A_n)$ with parameter $\beta$, this is
not possible. Therefore, we have the inequality $|\Gamma_i| \leq 2^\beta
(1+|s_i|)$.
Collecting the estimates above, we obtain
\[
	\card \Lambda \leq \sum_{i=1}^d | \Gamma_i | \lesssim_{\gamma,\beta} \sum_{i\neq \delta}
	(1+|s_i|),
\]
which is the desired estimate.
\end{proof}

\section{ Proof of Theorem~\ref{thm:main} }\label{sec:proof_main}
This section contains the proof of Theorem~\ref{thm:main}.
Therefore, fix an interval filtration $(\mathscr F_n)_{n\geq 0}$ on a
$d$-dimensional rectangle $I$ in standard
form (with $\mathscr F_0 = \{\emptyset, I\}$) that is  $k$-regular with
parameter $\gamma$ and
and direction $k$-regular with parameter $\beta$ and let $(f_n)$ be the
orthonormal system described in
Section~\ref{sec:rearr}.
In order to prove Theorem~\ref{thm:main} it is enough to prove, for arbitrary
integers $N$, all sequences $(\varepsilon_n)_{n\leq N}$ of signs
and all functions  $f = \sum_{n\leq N} a_n f_n$ with  $T_\varepsilon f = \sum_{n\leq N}
	\varepsilon_n a_n f_n$, the following inequality:
	\begin{equation}\label{eq:wt11}
		| \{ \sup_{M\leq N} | P_M (T_\varepsilon f) | >
		\lambda \}| \lesssim_{\gamma,\beta}
		\frac{\|f\|_1}{\lambda},\qquad \lambda>0,
	\end{equation}
	where $P_M$ denotes the orthogonal projector onto $\lin (f_n)_{n\leq
	M}$, which, by the discussion before Proposition~\ref{prop:proj_max}, is
	uniformly bounded on $L^1$.
	
	Fix the index $N$, a function $f = \sum_{n\leq N} a_n f_n$ and a
	positive number $\lambda$. If we let the index $\ell$ be such that
	$\mathscr F_\ell$ is associated to the function $f_N$, we assume without
	restriction that $N$ is chosen sufficiently large so that the $\sigma$-algebra associated
	to $f_{N+1}$ is $\mathscr F_{\ell +1}$.
	We also assume without restriction that $(\mathscr F_n)$ is of the
	form that $(\mathscr F_n)_{n>\ell}$ is a dyadic extension of $\mathscr
	F_\ell$.
	Based on this interval filtration $(\mathscr F_n)$ we consider the
	associated interval filtration $(\mathscr A_n)$ defined in
	Section~\ref{sec:rearr} so that the function $f_n$ is associated to the
	$\sigma$-algebra $\mathscr A_n=\mathscr A_n^{1}\otimes \cdots\otimes
	\mathscr A_{n}^{d}$ for every index $n$.

\subsection{A maximal function}
Let $U_{n}^j$ for $1\leq j\leq d$ be a union of $\ell_j$ neighboring atoms
$A_{n,1},\ldots, A_{n,\ell_j}$ of $\mathscr A_n^{j}$ for $1\leq
\ell_j \leq 3k_j$ so that the lengths of the leftmost atom $A_{n,1}$ and the
rightmost atom $A_{n,\ell_j}$ are
comparable to the length of $U_{n}^j$, i.e.
\begin{equation}\label{eq:bd}
	\min ( |A_{n,1}|, |A_{n,\ell_j}| ) \gtrsim_\gamma |U_{n}^j|
\end{equation}
and so that $U_{n}^j$ contains at least one
B-spline support of order $k_j$ in $\mathscr A_n^{j}$.
Observe that if $S$ is a B-spline support of order $k_j$ in  $\mathscr
A_n^{j}$ then there exists such a set $U_{n}^j$ with $U_{n}^j\supset S$ by the $k_j$-regularity of the
$\sigma$-algebra $\mathscr A_n^{j}$.
Let $\mathscr C_n$ be the collection of all $U_n= U_{n}^1 \times \cdots
\times U_{n}^d$ arising in this way. Let $a(N)$ be a sufficiently large integer
so that for each atom $A$ of
$\mathscr A_N$ and for every $t\in A$ there exists a set $B\in \mathscr
C_{a(N)}$  with $t\in B\subset A$. (This is possible since $ (\mathscr
F_n)_{n\geq\ell}$
	is a dyadic extension of $\mathscr F_\ell$.)
Set $\mathscr C = \cup_{n\leq a(N)} \mathscr C_n$ and 
define the maximal function 
\[
	\mathscr M_{\mathscr C}u(x) = \sup_{B\in\mathscr C, x\in B} \frac{1}{|B|} \int_B
	|u(t)|\dif t.
\]
For $B\in \mathscr C$, define $n(B)$ to be the smallest index $n$ so that $B\in
\mathscr C_n$.

It can be seen that $\mathscr M_{\mathscr C}u(x) \lesssim \mathscr M u(x)$ with
the maximal function $\mathscr M$ defined in \eqref{eq:max_fct}.
Indeed, let $B\in \mathscr C_n$ with $x\in B$ be arbitrary. 
Then, we can divide $B$ into at most $(3k_1)\cdots (3k_d)$ atoms $V_j$ of
$\mathscr A_n$ satisfying 
$|d_n (A_n(x), V_j)|_1 \leq \sum_{\delta=1}^d 3k_\delta=:C$. Since $\conv(V_j\cup
A_n(x)) \subset B$ for any $j$, we estimate
\begin{align*}
	\frac{1}{|B|} \int_B |u(t)|\dif t &\leq \sum_{j} \frac{1}{|\conv(V_j\cup
	A_n(x))|} \int_{V_j} |u(t)|\dif t \\
	&\leq q^{-C}\sum_{j} \frac{q^{|d_n (V_j,
	A_n(x))|_1 }}{|\conv(V_j\cup
	A_n(x))|} \int_{V_j} |u(t)|\dif t \lesssim  \mathscr M u(x).
\end{align*}
Therefore, by Theorem~\ref{prop:maximal}, $\mathscr M_{\mathscr C}$ is of weak
type $(1,1)$ as well.

\subsection{Decomposition of $f$}
We prove inequality \eqref{eq:wt11}
by splitting  the function $f$ using the maximal function $\mathscr M_{\mathscr
C}f$.
Assume that $\|f\|_1 \leq \lambda |I|/2$ since otherwise inequality
\eqref{eq:wt11} is clear.
Define $G_\lambda = \{\mathscr M_{\mathscr C}f > \lambda\}$. For $x\in
G_\lambda$, let $B(x)\in \mathscr C$ be chosen so that
$x\in B(x)$ and with
\begin{equation}\label{eq:condlambda}
	\frac{1}{|B(x)|} \int_{B(x)} |f(t)|\dif t > \lambda
\end{equation}
and also so that for all strictly larger sets $\mathscr C\ni\widetilde{B}\supset B(x)$,
we have the opposite inequality
\[
	\frac{1}{|\widetilde{B}|}\int_{\widetilde{B}} |f(t)|\dif t \leq \lambda.
\]
Note that $B(x)\neq I$ for any $x\in G_\lambda$ because of the assumption
$\|f\|_1 \leq \lambda|I|/2$.
Then, the collection of all those sets
$\{ B(x) : x\in G_\lambda\}$
covers the set $G_\lambda$. Let $(E_j)_j$ be an enumeration of the finitely many different
sets in the collection $\{ B(x) : x\in G_\lambda\}$. Those sets are not
necessarily disjoint, but we will show that 
\begin{equation}\label{eq:charfun}
	\sum_j \charfun_{E_j} \lesssim 1.
\end{equation}
To see this, let $t\in I$ be arbitrary and
we divide the family $\Gamma(t)= \{E_j : t\in E_j\}$ into a number of subcollections.
First, let for $1\leq \ell_j\leq 3k_j$ and $\ell = (\ell_1,\ldots,\ell_d)$
\[
	\Gamma_{\ell}(t) = \{ E = E^1 \times \cdots \times E^d \in\Gamma(t) :
		E^\delta \text{ consists of
		$\ell_\delta$ atoms of $\mathscr A_{n(E)}^{\delta}$ for all
	$\delta$} \}.
\]
Next we divide according to which atom from left to right the point $t$
belongs to. For each choice of $\ell_1,\ldots,\ell_d$, and for each $1\leq
m_\delta
\leq \ell_\delta$ for $\delta=1,\ldots,d$ we define the collection of all sets $E\in
\Gamma_{\ell}(t)$ so that the point $t$ belongs to the $m_\delta$th
atom of $E^\delta$ from left to right as $\Gamma_{\ell,m}(t)$ writing $m=
(m_1,\ldots,m_d)$.
If two sets $E,F\in\Gamma(t)$ are contained in the same collection
$\Gamma_{\ell,m}(t)$, by the nestedness of the
$\sigma$-algebras $(\mathscr A_n)$ we must have that either $E$ is contained
in $F$ or vice versa. But since the sets $E,F$ are chosen
maximally under condition \eqref{eq:condlambda}, we must have $E=F$.
Therefore, each collection $\Gamma_{\ell,m}(t)$ 
only consists of at most one set and since the number of
collections $\Gamma_{\ell,m}(t)$ is bounded by some constant depending on
$k=(k_1,\ldots,k_d)$, we have
proven \eqref{eq:charfun}. 

Next we disjointify the collection $(E_j)$ and set
\[
	V_j = E_j \setminus \bigcup_{i<j} E_i.
\]
Obviously $(V_j)$ consists of disjoint sets, $\cup_j V_j = G_\lambda$
and $V_j\subset E_j$ for each $j$.
Based upon this disjoint decomposition of $G_\lambda$, we split the function
$f$ into the following parts:
\begin{align} \label{eq:hdef}
	h &= f \cdot \charfun_{G_\lambda^c} + \sum_{j}
	Q_{E_j}(f\charfun_{V_j}),\\ 
	\label{eq:gdef}
	g &= f-h = \sum_{j} \big(f\charfun_{V_j} - Q_{E_j}(f\charfun_{V_j})
	\big),
\end{align}
where the operator $Q_{E_j}$ is given as follows.
For fixed $j\geq 1$, we have  $E_j\in \mathscr C_n$ with $n= n(E_j)$.
Then, let $Q_{E_j}$ be the orthogonal projection operator
onto the spline space $S_k(\mathscr A_n \cap E_j)$.
Writing $T_\varepsilon g = \sum_{n\leq N} \varepsilon_n \langle g,f_n\rangle
f_n$ and $T_\varepsilon h = \sum_{n\leq N} \varepsilon_n \langle h,f_n\rangle
f_n$, we obtain
$T_\varepsilon f = T_\varepsilon h + T_\varepsilon g$ and thus
\begin{equation}\label{eq:splitf}
	| \{ \sup_{M\leq N} |P_M( T_\varepsilon f) | > \lambda \} | \leq  | \{
		\sup_{M\leq N} |P_M(T_\varepsilon h)| >
	\lambda/2 \} | + | \{ \sup_{M\leq N} |P_M(T_\varepsilon g)| > \lambda/2 \} |.
\end{equation}
\subsection{The function $h$} We start with the estimate
\begin{equation}\label{eq:h0}
	| \{ \sup_{M\leq N} |P_M(T_\varepsilon h)| > \lambda/2 \} | \leq
	\frac{4}{\lambda^2} \big\|
	\sup_{M\leq N} |P_M(T_\varepsilon h)|\big\|_2^2. 
\end{equation}
Since the maximal function of the operators $P_M$ is bounded uniformly on $L^2$ by
\eqref{eq:estPn} and Theorem~\ref{prop:maximal},  we estimate further
\begin{equation}\label{eq:h1}
	| \{ \sup_{M\leq N} |P_M(T_\varepsilon h)| > \lambda/2 \} | \lesssim
	\frac{4}{\lambda^2} \|
	T_\varepsilon h\|_2^2 \leq \frac{4}{\lambda^2}\| h\|_2^2,
\end{equation}
where the last equation follows from the orthogonality of the functions $(f_n)$. Therefore, it suffices to estimate
the $L^2$-norm of $h$.
We first estimate
$|f|$ pointwise a.e. on $G_\lambda^c$. Since $f\in S_N$ 
we let for $t\in G_\lambda^c$ 
the atom $A$ in $\mathscr A_N$ be such that
$t\in A$. 
By definition of $a(N)$ and $\mathscr C$, there exists $B
\in \mathscr C$ with $t\in B\subset
A$. Since $t\in G_\lambda^c$, we know that 
\[
	\frac{1}{|B|}\int_B |f(s)|\dif s \leq \lambda.
\]
Since $f$ is a polynomial on $B\subset A$, we invoke Remez' inequality
(Corollary~\ref{cor:remez}) to deduce
$|f(t)|\leq \|f\|_{L^\infty(B)} \lesssim
\lambda$. This argument shows that $|f|\lesssim \lambda$ a.e. on $G_\lambda^c$
and allows us to estimate further
\begin{equation}
	\label{eq:h}
\begin{aligned}
	\|h\|_2^2 &= \int_{G_\lambda^c} |f|^2 + \int \Big | \sum_{j}
	Q_{E_j}(f\charfun_{V_j})\Big|^2 \lesssim \lambda \int_{G_\lambda^c} |f| + \int
	\Big( \sum_{j}
	|Q_{E_j}(f\charfun_{V_j})|\Big)^2 \\
	&\lesssim \lambda\int_{G_\lambda^c} |f| +
	\sum_{j} \int | Q_{E_j}(f\charfun_{V_j}) |^2,
\end{aligned}
\end{equation}
where the last inequality follows from the fact that the non-negative function $u_j  :=
|Q_{E_j}(f\charfun_{V_j})|$ has support contained in $E_j$.
Indeed, let $t\in I$ and let  
$j_1(t),\ldots,j_m(t)$ be an enumeration of the indices $j$ so that $t$ is
contained in the support of $u_{j}$. By
\eqref{eq:charfun}, we know that $m\lesssim 1$. Therefore,
\begin{align*}
	\Big( \sum_j u_j(t) \Big)^2 = \big( u_{j_1}(t) + \cdots + u_{j_m}(t)
	\big)^2
	\lesssim \big( \max_\ell u_{j_\ell}(t) \big)^2 \leq \sum_{\ell}
	u_{j_\ell}(t)^2 \leq  \sum_j u_j(t)^2.
\end{align*}
Next, we will show that, for all $j$, we have the estimate 
\begin{equation}\label{eq:localproj}
	\int |Q_{E_j}(f\charfun_{V_j})|^2 \lesssim_\gamma \lambda^2 |E_j|.
\end{equation}
Indeed, setting $n=n(E_j)$, let $(N_i)$ be the B-spline basis of $S_k(\mathscr A_n\cap E_j)$ and denote by 
$(N_i^*)$ its dual basis. 
Since the linear span of $(N_i)$ is 
the range of the operator  $Q_{E_j}$,
\[
	Q_{E_j}(f\charfun_{V_j}) = \sum_i \langle f\charfun_{V_j}, N_i\rangle
	N_i^*.
\]
By the properties of the sets $E_j\in \mathscr C_n$ (in particular by the choice
\eqref{eq:bd} of boundary intervals) and the $k$-regularity of the filtration $(\mathscr A_n)$, we get that the dual functions $N_i^*$ satisfy
the estimate
\[
	|N_i^*(t)| \lesssim_\gamma \frac{1}{|E_j|},\qquad t\in E_j,
\]
by inequality \eqref{eq:mainestimate}. This implies
\begin{equation}
	\label{eq:est1}
	\begin{aligned}
		\int_{E_j} |Q_{E_j}(f\charfun_{V_j})|^2 &\lesssim  \sum_i |\langle
	f\charfun_{V_j}, N_i\rangle|^2 \int_{E_j} N_i^*(t)^2\dif t  \\
	&\lesssim_\gamma \sum_i \Big(
	\int_{E_j} |f| \Big)^2 \frac{1}{|E_j|} \lesssim \Big(
	\int_{E_j} |f|
	\Big)^2 \frac{1}{|E_j|}
\end{aligned}
\end{equation}
since the sum over $i$ only contains a constant number of terms (depending on
the order of the splines $k=(k_1,\ldots, k_d)$).
Recall $n=n(E_j)$. 
 Then, let $A\in\mathscr C_{n-1}$ so that $E_j\subset A$.
Observe that by definition of $n(E_j)$, $A=A^1\times \cdots \times A^d$ is a 
strict superset of $E_j=E_j^1\times \cdots \times E_j^d$ in the sense that there
exists precisely one coordinate $\delta=1,\ldots,d$ so that $A^\delta$ is a
strict superset of $E_j^\delta$. This means that one of the atoms of $\mathscr
A_n^\delta$ contained in $E_j^\delta$ is not an atom in $\mathscr
A_{n-1}^\delta$. Nevertheless, $A^\delta$ is a subset of the union of a constant
(depending on $k_\delta$) number of neighbouring B-spline supports in $\mathscr
A_n^\delta$, at least one of them being a subset of $E_j^\delta$. Therefore, by
$k_\delta$-regularity of the $\sigma$-algebra $\mathscr A_n^\delta$, we obtain
$|A^\delta|\lesssim_\gamma |E_j^\delta|$ and therefore, $|A| \lesssim_\gamma |E_j|$.
By the maximality of  $E_j$ under condition \eqref{eq:condlambda}, we infer 
\[
	\frac{1}{|E_j|} \int_{E_j} |f(t)|\dif t \lesssim_\gamma \frac{1}{|A|}\int_{A}
	|f(t)|\dif t \leq \lambda.
\]
Therefore, we continue the estimate in \eqref{eq:est1} and write
\[
	\int |Q_{E_j}(f\charfun_{V_j})|^2 \lesssim_\gamma \frac{1}{|E_j|}
	\Big( \int_{E_j} |f(t)|\dif t \Big)^2  
	 \lesssim_\gamma \lambda^2 |E_j|,
\]
which shows inequality \eqref{eq:localproj}.
Summing over $j$ yields 
\[
	\sum_{j} \int | Q_{E_j}(f\charfun_{V_j}) |^2 \lesssim_\gamma \lambda^2
	|G_\lambda|\lesssim \lambda \|f\|_1
\]
by inequality \eqref{eq:charfun} and 
the weak type estimate for the maximal function $\mathscr M_{\mathscr C}$. Inserting this
inequality in the estimate \eqref{eq:h} for the $L^2$ norm of the function $h$
yields
$\| h \|_2^2 \lesssim_\gamma\lambda\| f\|_1$,
which, together with \eqref{eq:h0} and \eqref{eq:h1}, gives the weak type estimate
\begin{equation}\label{eq:hfinal}
	| \{ \sup_{M\leq N} |P_M(T_\varepsilon h)| > \lambda/2 \}| \lesssim_\gamma \frac{\|f\|_1}{\lambda}.
\end{equation}

\subsection{The function $g$}
Let $j$ be arbitrary and let $n=n(E_j)$ and thus $E_j= E_j^{1}\times \cdots \times
E_j^{d}\in \mathscr C_n$. 
We know that for all $\delta=1,\ldots,d$, the set  $E_j^{\delta}$ is a union of $\ell_\delta$ atoms in $\mathscr
A_n^{\delta}$ for some $1\leq \ell_\delta\leq 3k_\delta$.
Then,
define $L_j^{\delta}$ and $H_j^{\delta}$ to be the union of at most $5k_\delta$ and
at most $7k_\delta$ neighboring atoms of
$\mathscr A_{n}^{\delta}$ respectively so that 
between $(H_j^\delta)^c$ and $L_j^\delta$ as well as between $(L_j^\delta)^c$
and $E_j^\delta$ are $k_\delta$ atoms of $\mathscr A_n^\delta$.
By $k$-regularity of $(\mathscr A_n)$, this implies that
the distance between $(L_j^{\delta})^c$
and $E_j^{\delta}$ as well as the distance between $(H_j^{\delta})^c$ and
$L_j^{\delta}$ is  $\gtrsim_\gamma |E_j^{\delta}|$ and, moreover,
$|H_j^\delta| \lesssim_\gamma |E_j^\delta|$.
Then, set $L_j = L_j^{1}\times \cdots \times L_j^{d}$ and $H_j =
H_j^{1}\times \cdots \times H_j^{d}$.
The letters $L$ and $H$ are chosen here to indicate that $L_j$ and $H_j$ are
large and huge  versions
of $E_j$, respectively.

Next, set $\widetilde{G_\lambda} = \cup_j H_j$. Then, we estimate the function
$g$ as follows:
\begin{equation}\label{eq:split_maximal}
	\big| \big\{ \sup_{M\leq N} |P_M(T_\varepsilon g)| > \lambda/2 \big\}
	\big| \leq |\widetilde{G_\lambda}| +
	\big|\big\{ t\in \widetilde{G_\lambda}^c : \sup_{M\leq N} |P_M(T_\varepsilon g)(t)| > \lambda/2
\big\}\big|,
\end{equation}
The term $|\widetilde{G_\lambda}|$ can be estimated by $|G_\lambda|$ if we
use inequality \eqref{eq:charfun}:
\begin{equation}\label{eq:esttilde}
	|\widetilde{G_\lambda}| \leq \sum_j |H_j| \lesssim_\gamma \sum_j |E_j| \lesssim
	|G_\lambda| \lesssim \frac{ \|f\|_1}{\lambda}.
\end{equation}
where the last inequality follows from the fact that the
maximal function $\mathscr M_{\mathscr C}$ is of weak type $(1,1)$.

Now we
come to the second term of \eqref{eq:split_maximal}, which we estimate as
follows:
\begin{equation}\label{eq:weakTg}
\begin{aligned}
	|\{ t\in \widetilde{G_\lambda}^c : \sup_{M\leq N} |P_M(T_\varepsilon g)(t)| > \lambda/2 \}| &\leq
	\frac{2}{\lambda} \Big\| \sup_{M\leq N} |P_M(T_\varepsilon g)| 
\Big\|_{L^1(\widetilde{G_\lambda}^c)} \\ &= \frac{2}{\lambda} \Big\|\sup_{M\leq
N} \big| \sum_j \sum_{n\leq M}
\varepsilon_n\langle g_j,f_n\rangle f_n \big|
\Big\|_{L^1(\widetilde{G_\lambda}^c)} \\
&\leq \frac{2}{\lambda} \sum_j \Big\| \sup_{M\leq N} \sum_{n\leq M} |\langle g_j,f_n\rangle| \cdot
|f_n|\Big\|_{L^1(\widetilde{G_\lambda}^c)} \\
&=\frac{2}{\lambda} \sum_j \Big\| \sum_{n\leq N} |\langle g_j,f_n\rangle| \cdot
|f_n| \Big\|_{L^1(\widetilde{G_\lambda}^c)}
\end{aligned}
\end{equation}
with $g_j = f\charfun_{V_j} - Q_{E_j}(f\charfun_{V_j})$.
We will, for fixed index $j$, show the inequality
\begin{equation}\label{eq:crucialtoshow}
\Big\| \sum_n |\langle g_j,f_n\rangle| \cdot
|f_n|\Big\|_{L^1(H_j^c)} \lesssim_{\gamma,\beta} \| g_j \|_{L^1(E_j)}.
\end{equation}
If we know \eqref{eq:crucialtoshow}, we can continue estimate \eqref{eq:weakTg},
since $\widetilde{G_\lambda}^c \subset H_j^c$ for any fixed index~$j$, and
therefore,
\begin{align*}
	|\{ t\in \widetilde{G_\lambda}^c : \sup_{M\leq N} |P_M(T_\varepsilon g)(t)| > \lambda/2 \}|
\lesssim_{\gamma,\beta}
\frac{1}{\lambda} \sum_j \| g_j \|_{L^1(E_j)}.
\end{align*}
By the uniform $L^1$ boundedness of the operator $Q_{E_j}$ (a consequence of
Shadrin's theorem~\ref{thm:shadrin}), we infer $\| g_j
\|_{L^1} \leq C \|f\|_{L^1(V_j)}$, which, together with the latter display and
the disjointness of the sets $(V_j)$,
gives us 
\[
	|\{ t\in \widetilde{G_\lambda}^c : \sup_{M\leq N} |P_M(T_\varepsilon g)(t)| > \lambda/2 \}|
	\lesssim_{\gamma,\beta} \frac{1}{\lambda} \| f\|_{L^1}.
\]
Combining this inequality with \eqref{eq:split_maximal}, \eqref{eq:hfinal},
\eqref{eq:esttilde}, and
\eqref{eq:splitf}  yields the conclusion of Theorem~\ref{thm:main}. Therefore,
we continue with the proof of \eqref{eq:crucialtoshow}.

Observe first that in order to show \eqref{eq:crucialtoshow}, we restrict the
summation to $n> n(E_j)$ where we recall that $n(E_j)$ is the 
smallest index $m$ so that $E_j \in \mathscr C_m$. This is possible, since for $n
\leq n(E_j)$, the function $f_n|_{E_j}$ is contained in the range of
the operator $Q_{E_j}$ (which is the spline space $S_k( E_j \cap \mathscr
A_{n(E_j)})$) and this implies
that $\langle g_j, f_n\rangle = 0$ due to the defining equation $g_j = f\charfun_{V_j} -
Q_{E_j}(f\charfun_{V_j})$.
We slightly change the language, fix the index $j$ and write $E= E_j$,
$L=L_j$, $H=H_j$ and $b$
for a generic function supported on $E$. Thus \eqref{eq:crucialtoshow} is
implied by the estimate
\[
	\Big\| \sum_{n> n(E)} |\langle b,f_n\rangle| \cdot
	|f_n|\Big\|_{L^1(H^c)} \lesssim_{\gamma,\beta} \|b\|_{L^1(E)}.
\]
Divide the index set $\{n> n(E)\}$ into the
two parts
\[
	\Gamma_1 = \{ n> n(E) : J_n \subset L^c \},\qquad
	\Gamma_2 = \{ n> n(E) : J_n \subset L \},
\]
where we recall that $J_n$ is the characteristic interval of the function $f_n$
defined in Section~\ref{sec:tensor_def}.
Since $L$ is a union of atoms in $\mathscr A_{n(E)}$ and $J_n$ (for $n>n(E)$) is
an atom in the finer $\sigma$-algebra $\mathscr F_n$,
we have $\{ n> n(E) \} =
\Gamma_1\cup \Gamma_2$ with a disjoint union.

\textsc{Case 1:} First consider the case $n\in\Gamma_1=\{ n> n(E) : J_n
\subset L^c \}$ and use the pointwise estimate \eqref{eq:pw} and its consequence
$\| f_n \|_{L^1} \lesssim |J_n|^{1/2}$ for the functions $f_n$ to deduce
(denoting $d_n = d_{\mathscr A_n}$)
\begin{equation}\label{eq:b_est}
\begin{aligned}
	\Big\| \sum_{n\in \Gamma_1} |\langle b,f_n\rangle| \cdot
	|f_n|\Big\|_{L^1(H^c)} &\leq \sum_{n\in\Gamma_1}
	|\langle b,f_n\rangle | \| f_n \|_{L^1} \\
	&\lesssim \sum_{n\in\Gamma_1} \sum_{\substack{\text{$B$ atom of
	$\mathscr A_n$:} \\
B\subset E}} 
	\frac{q^{|d_n(J_n,B)|_1} |J_n| } { |\conv(J_n, B)| }
\int_B |b(y)|\dif y \\
&= \sum_{s\in\mathbb Z^d} q^{|s|_1} \int_E\Big(\sum_{n\in\Gamma_1}\sum_{\substack{\text{$B$
atom of $\mathscr A_n$:} \\
	B\subset E,\ d_n(J_n,B) = s}}
 \frac{ |J_n|\charfun_B(y) } { |\conv(J_n, B)| }\Big)
 |b(y)|\dif y.
\end{aligned}
\end{equation}
Let $s\in \mathbb Z^d$ and $y\in E$ be fixed. Define  $\Gamma_{1,y,s}$ to be the set of all 
$n\in \Gamma_1$ so that there exists an atom $B$ of $\mathscr A_n$ with
$y\in B\subset E$ and $d_n(J_n,B) = s$. Since $s\in \mathbb Z^d$ is fixed, this atom $B$
is given uniquely by the index $n\in \Gamma_{1,y,s}$ and we denote
$B_n = B$.
Now, split the index set $\Gamma_{1,y,s}$ into the not necessarily disjoint union
\[
	\Gamma_{1,y,s} = \bigcup_{\delta=1}^d \Gamma_{1,y,s}^{\delta},\qquad
	\text{with }\Gamma_{1,y,s}^{\delta} = \{ n\in \Gamma_{1,y,s} : J_{n}^{\delta} \subset
(L^{\delta})^c \}.
\]
We fix the parameter $\delta \in \{1,\ldots,d\}$ 
and proceed to estimate the sum 
\begin{equation}\label{eq:sumtoestimate}
	\sum_{n\in\Gamma_{1,y,s}^{\delta}} \frac{|J_{n}|}{|\conv(J_{n},B_{n})|}
	\leq \sum_{n\in\Gamma_{1,y,s}^{\delta}}
	\frac{|J_{n}^\delta|}{|\conv(J_{n}^\delta,B_{n}^\delta)|}
	= 
	\sum_{n\in\Gamma_{1,y,s}^{\delta}} \int_{J_{n}^\delta}
	\frac{1}{|\conv(J_{n}^\delta,B_{n}^\delta)|}\dif t.
\end{equation}
Let $x$ be the endpoint of $E^\delta$ that is closest to the sets
$J_{n}^\delta$ for $n\in \Gamma_{1,y,s}^\delta$. Then, observe that
$|x-t| \leq |\conv(J_{n}^\delta, B_{n}^\delta)|$ for $t\in J_{n}^\delta$.
Recall that $J_{n}^\delta\subset (L^\delta)^c$ and
$B_{n}^\delta\subset E^\delta$ for all $n\in \Gamma_{1,y,s}^{\delta}$.
Therefore, $|x-t|\geq c |E^\delta|$ for $t\in J_{n}^\delta$ and some constant $c$
depending only on $k$ and $\gamma$.
Moreover, by $k_\delta$-regularity of the filtration in
direction $\delta$, we have $ |\conv(J_{n}^\delta,B_{n}^\delta)| \leq C
\gamma^{|s_\delta|}
|E^\delta|$ for some absolute constant $C$ (we can assume
without restriction that $\gamma \geq 2$).
Those estimates yield
\[
	c|E^\delta| \leq |x-t| \leq C\gamma^{|s_\delta|} |E^\delta|,
	\qquad t\in J_{n}^\delta.
\]
Let $\Lambda$ be a set of indices $n\in \Gamma_{1,y,s}^{\delta}$ so that for
$i,j\in\Lambda$ with $i < j$ we have $J_{i}^\delta \supseteq
J_{j}^\delta$.
We invoke Lemma \ref{lem:comb} with the setting $C_n = B_{n}$ for $n\in\Lambda$ to deduce that the
cardinality of $\Lambda$ is  $\lesssim_{\gamma,\beta} \sum_{j\neq \delta} (1+|s_j|)$.
Those observations imply
\begin{align*}
\sum_{n\in\Gamma_{1,y,s}^{\delta}} \int_{J_{n}^\delta}
	\frac{1}{|\conv(J_{n}^\delta,B_{n}^\delta)|}\dif t &\leq 
	\sum_{n\in\Gamma_{1,y,s}^{\delta}} \int_{J_{n}^\delta}
	\frac{1}{|x-t|}\dif t \\
	&\lesssim_{\gamma,\beta}
	\Big(\sum_{j\neq \delta} (1+|s_j|)\Big)
	\int_{c|E^\delta|}^{C\gamma^{|s_\delta|} |E^\delta|}
	\frac{\dif u}{u} \lesssim_{\gamma} |s_\delta|\cdot  \sum_{j\neq \delta} (1+|s_j|).
\end{align*}
Inserting this estimate, combined with \eqref{eq:sumtoestimate},
in the last line of \eqref{eq:b_est} and summing a geometric series, we obtain that
\[
\Big\| \sum_{n\in \Gamma_1} |\langle b,f_n\rangle| \cdot
|f_n|\Big\|_{L^1(H^c)}  \lesssim_{\gamma,\beta}  \int_E |b(t)|\dif t.
\]

\textsc{Case 2: }
Now we consider $n\in\Gamma_2 = \{ n> n(E) : J_n \subset L \}$.
Estimate \eqref{eq:pw}  with the notation $d_n = d_{\mathscr
A_n}$ gives 
\begin{equation}\label{eq:case2}
\begin{aligned}
	\Big\| \sum_{n\in \Gamma_2}& |\langle b,f_n\rangle| \cdot 
	|f_n|\Big\|_{L^1(H^c)} \leq \sum_{n\in\Gamma_2}
	|\langle b,f_n\rangle | \| f_n \|_{L^1(H^c)} \\
	&\leq \sum_{n\in\Gamma_2} \sum_{\substack{\text{$A,B$ atom of $\mathscr
	A_n$:} \\
	B\subset E,\  A\subset H^c}
	} 
	\frac{q^{|d_n(B,J_n)|_1 + |d_n(A,J_n)|_1} |J_n|\cdot |A| } { |\conv(J_n,
	B)|\cdot |\conv(J_n,A)| }
\int_B |b(y)|\dif y \\
&= \sum_{r,s\in\mathbb Z^d} q^{|r|_1 + |s|_1}\int_E \Big(\sum_{n\in\Gamma_2}\sum_{\substack{\text{$A,B$ atom of $\mathscr
	A_n$:} \\
	B\subset E,\  A\subset H^c, \\ d_n(A, J_n) = r,\ 
d_n(A,B) = s} }
	\frac{ |J_n|\cdot |A|\cdot \charfun_{B}(y) } { |\conv(J_n,
	B)|\cdot |\conv(J_n,A)| } \Big)
 |b(y)|\dif y.
\end{aligned}
\end{equation}
For fixed $r,s\in \mathbb Z^d$ and $y\in E$, we let $\Gamma_{2,y,(r,s)}$ be the set of
all $n\in \Gamma_2$ so that there exist two atoms $A,B$ of $\mathscr A_n$ with
$y\in B\subset E$, $A\subset H^c$ and $d_n(A,J_n) = r$, $d_n(A,B) = s$.
For fixed $r,s\in\mathbb Z^d$, those atoms $A,B$ are given uniquely by the index
$n\in \Gamma_{2,y,(r,s)}$ and are denoted by $A_n,B_n$ respectively.
Note that if $n\in \Gamma_{2,y,(r,s)}$ then we know that $d_n(J_n, B_n) =
s-r$.
Split the index set $\Gamma_{2,y,(r,s)}$ further into the (not necessarily disjoint) subcollections
\[
	\Gamma_{2,y,(r,s)}^{\delta} = \{ n\in \Gamma_{2,y,(r,s)} :
	A^{\delta}
	\subset (H^{\delta})^c \},\qquad
	\delta\in\{1,\ldots,d\}.
\]
Fixing the parameter $\delta\in \{1,\ldots,d\}$
we estimate the sum
\begin{equation}
	\label{eq:deltasplit}
	 \sum_{n\in
		 \Gamma_{2,y,(r,s)}^{\delta}} \frac{ |J_n| |A_n| }{ |
		\conv(J_n, B_n)| \cdot | \conv(J_n,A_n) |} \leq 
\sum_{n\in
	\Gamma_{2,y,(r,s)}^{\delta}} \frac{ |J_n^\delta| |A_n^\delta| }{ |
		\conv(J_n^\delta, B_n^\delta)| \cdot |
		\conv(J_n^\delta,A_n^\delta) |}.
	\end{equation}
Consider the indices of all maximal sets $A_n^\delta$ by setting
\[
	\Lambda = \{ n \in \Gamma_{2,y,(r,s)}^{\delta} : \text{ there is no
		$m\in \Gamma_{2,y,(r,s)}^{\delta}$ with $m < n$ and
		$A_m^\delta \supseteq A_n^\delta$} \}.
\]
Moreover, for $m\in\Lambda$, define
\[
	\Lambda_m = \{ n\in \Gamma_{2,y,(r,s)}^{\delta} : A_n^\delta
	\subseteq A_m^\delta\}
\]
and split the sum on the right hand side of \eqref{eq:deltasplit} into 
\begin{equation}\label{eq:deltamax}
	\sum_{m\in\Lambda} \sum_{n\in\Lambda_m }\frac{ |J_n^\delta|
		|A_n^\delta| }{ |
			\conv(J_n^\delta, B_n^\delta)| \cdot |
			\conv(J_n^\delta,A_n^\delta) |}.
		\end{equation}
		Fix $m\in \Lambda$.
If  $n\in\Lambda_m$ we have $A_n^\delta \subset A_m^\delta$, and thus we estimate $|A_n^\delta|\leq
|A_m^\delta|$ and also
$|\conv(J_m^\delta,A_m^\delta)| \lesssim_\gamma
|\conv(J_n^\delta,A_n^\delta)|$ since already in the $\sigma$-algebra $\mathscr
A_{m}^\delta$,
we have at least $ k_\delta$ atoms between $J_m^\delta$ and $A_m^\delta$ (recall
$A_m^\delta\subset (H^\delta)^c$, $m>n(E)$,
 and $J_n^\delta\subset L^\delta$) 
 and $k_\delta$-regularity in direction $\delta$ gives this inequality. Therefore, we can
estimate the  sum in \eqref{eq:deltamax} from above by
\begin{equation}\label{eq:toshow_double}
	\sum_{m\in\Lambda}
	\frac{|A_m^\delta|}{|\conv(J_m^\delta,A_m^\delta)|}
	\sum_{ n\in\Lambda_m}\frac{ |J_n^\delta|  }{ |
		\conv(J_n^\delta, B_n^\delta)| }.
	\end{equation}
	Fix $m\in\Lambda$, fix a B-spline support $\Delta$ of order $k_\delta$ in the
	$\sigma$-algebra $\mathscr A_m^\delta$ between
	$A_m^\delta$ and $L^\delta$.
	For all $n\in\Lambda_m$,
	$\Delta$ is a B-spline support of order $k_\delta$ in the $\sigma$-algebra $\mathscr
	A_n^\delta$ as well since the number of atoms of $\mathscr A_n^\delta$
	between $A_n^\delta$ and $B_n^\delta$ is constant and $A_n^\delta\subset
	A_m^\delta$ and also the sets $B_n^\delta$ are decreasing.
	If $n\in \Lambda_m$, we know
	that between $A_n^\delta$ and $J_n^\delta$ we have $|r_\delta|$
	atoms of $\mathscr A_n^\delta$ and this means $|d_n^\delta
	(D, J_n^\delta)| \leq |r_\delta|$ for any atom $D\subset \Delta$ in
	$\mathscr A_n^\delta$, $n\in\Lambda_m$ (using the notation $d_n^\delta =
	d_{\mathscr A_n^\delta}$).
	This implies by $k_\delta$-regularity of the $\sigma$-algebra $\mathscr
	A_n^\delta$
\begin{equation}\label{eq:reg}
	\frac{\gamma^{-|r_\delta|}|\Delta|}{k_\delta} \leq  |J_{n}^\delta| \leq
	\gamma^{|r_\delta|}
	|\Delta|.
\end{equation}
Now we consider two subcases relating the values of $r$ and $s$ for the
analysis of the inner sum in \eqref{eq:toshow_double} for fixed $m\in\Lambda$.
We remark that by definition of $r$ and $s$ and the location of $A_n,B_n,J_n$, 
the sign of $r_\delta$ is the same as the sign of $s_\delta$.

	\textsc{Case 2a: }$|r_\delta|<|s_\delta|$:
	In this case, the set $J_n^\delta$ is strictly between
	$A_n^\delta$ and $B_n^\delta$ for  all $n\in \Lambda_m$. 
	This implies that
	since $d_n^\delta(A_n^\delta,J_n^\delta) = r_\delta$ and
	$d_n^\delta(J_n^\delta,B_n^\delta)= s_\delta-r_\delta$ are both constant
	for $n\in \Lambda_m$,
	that the sets $J_n^\delta$ have to coincide for all $n\in \Lambda_m$. 
	Set $C_n = B_n^{1}\times \cdots \times B_n^{\delta-1}
\times A_n^{\delta}\times B_n^{\delta+1}\times \cdots \times B_n^{d}$ for
$n\in \Lambda_m$ and apply Lemma~\ref{lem:comb} to deduce that  the cardinality
of $\Lambda_m$ is $\lesssim_{\gamma,\beta} \sum_{j\neq \delta} (1+|s_j - r_j|)$.
This gives the estimate
\[
	\sum_{n\in\Lambda_m }\frac{ |J_n^\delta|}
	{|\conv(J_n^\delta, B_n^\delta)| } \lesssim_{\gamma,\beta}  \sum_{j\neq \delta} (1 + | s_j  -r_j|).
\]

	\textsc{Case 2b: }$|r_\delta|\geq|s_\delta|$:
	Here, $B_n^\delta$ is between $A_n^\delta$ and
	$J_n^\delta$ for all $n\in\Lambda_m$.
Let $x$ be the point in $B_m^\delta$ closest to $A_m^\delta$. For
$n\in\Lambda_m$, let $I_n\subset J_n^\delta$ be an interval such that
$|I_n|=|J_n^\delta|/2$ and $\dist (x,I_n) \geq |J_n^\delta|/2$.
Then 
\[
\sum_{n\in\Lambda_m}\frac{ |J_n^\delta|  }{ |
		\conv(J_n^\delta, B_n^\delta)| }
	=	
		2 \sum_{n\in\Lambda_m}
		\int_{I_n}
		\frac{1}{|\conv(J_n^\delta,B_n^\delta)|} \dif t.
\]
 Note that
$x$ is also an endpoint of $B_n^\delta$ for all
 $n\in\Lambda_m$ and thus, if  $t\in I_n$,
 \begin{equation}\label{eq:x_minus_t}
	|J_n^\delta| /2\leq |x-t| \leq |\conv(J_n^\delta,B_n^\delta)|,\qquad
	n\in\Lambda_m.
\end{equation}
Let $\Omega \subset \Lambda_m$ be so that for $i,j\in\Omega$
with $i<j$ we have $J_i^\delta \supseteq J_j^\delta$. Then we apply
Lemma~\ref{lem:comb} with $C_n = B_n^{1}\times \cdots \times B_n^{\delta-1}
\times A_n^{\delta}\times B_n^{\delta+1}\times \cdots \times B_n^{d}$ for
$n\in \Omega$ to deduce that the cardinality of $\Omega$ is 
$\lesssim_{\gamma,\beta} \sum_{j\neq \delta} (1+|s_j - r_j|)$.
Therefore we estimate further 
\[
	\sum_{n\in\Lambda_m} \int_{I_n}
	\frac{1}{|\conv(J_n^\delta,B_n^\delta)|} \dif t \lesssim_{\gamma,\beta}
	\Big(\sum_{j\neq \delta} (1+|s_j - r_j|) \Big)
	\int_{\cup_{n} I_n}
	\frac{1}{|x-t|}\dif t,
\]
which, by inequalites \eqref{eq:x_minus_t} and \eqref{eq:reg}, is smaller than
\[
	\Big(\sum_{j\neq \delta} (1+|s_j - r_j|)\Big)
	\int_{\gamma^{-|r_\delta|}
|\Delta|/(2k_\delta)}^{K\gamma^{|r_\delta|} |\Delta|}
\frac{\dif u}{u}\lesssim_{\gamma,\beta}
	\Big( \sum_{j\neq \delta}(1+ |s_j - r_j|)\Big) \cdot |r_\delta|, 
\]
for some absolute constant $K$.

Thus, we come back to \eqref{eq:toshow_double} and combine  the results of
subcases \textsc{2a} and \textsc{2b} to
obtain
\[
	\sum_{m\in\Lambda}
	\frac{|A_m^\delta|}{|\conv(J_m^\delta,A_m^\delta)|}
	\sum_{ n\in\Lambda_m}\frac{ |J_n^\delta|  }{ |
		\conv(J_n^\delta, B_n^\delta)| }
		\lesssim_{\gamma,\beta} \Big( \sum_{j\neq \delta} (1+|s_j - r_j|)\Big)
		|r_\delta|
	\sum_{m\in\Lambda}
	\frac{|A_m^\delta|}{|\conv(J_m^\delta,A_m^\delta)|}.
\]
The next thing is to 
estimate the latter sum in terms of $r$ and $s$.
This can be done as follows. 

Denote by $x$ the endpoint of $L^\delta$ that is closest
to the sets $A_m^\delta$, $m\in\Lambda$.
Since for all $m\in\Lambda$, $J_m^\delta$ is a subset of $L^\delta$ and
$A_m^\delta$ is a subset
of $(H^\delta)^c$, the distance between $x$ and $t$ is greater than
$c |E^\delta|$ for all $t\in A_m^\delta$ and some constant $c$ depending only on $k,\gamma$. Using 
$|d_m^\delta(A_m^\delta, J_m^\delta)| = |r_\delta|$ for all $m\in\Lambda$ and $k_\delta$-regularity
of $(\mathscr A_n^\delta)$, we obtain

\[
c|E^\delta|\leq |x-t| \leq  |\conv(J_m^\delta, A_m^\delta)| \leq C
\gamma^{|r_\delta|} |E^\delta|,
\qquad t\in A_m^\delta
\]
for some absolute constant $C$.
This implies, since the maximal sets $A_{m}^\delta$, $m\in\Lambda$, are disjoint,
\[
	\sum_{m\in \Lambda} \frac{|A_m^\delta|}{|\conv
		(J_m^\delta, A_m^\delta)|} \leq
		\sum_{m\in\Lambda} \int_{A_m^\delta} \frac{1}{|x-t|} \dif t
\leq \int_{c|E^\delta|}^{C\gamma^{|r_\delta|} |E^\delta|}
	\frac{1}{u}\dif u \lesssim_{\gamma} |r_\delta|.
\]
Thus
\[	 
	\sum_{m\in\Lambda}
	\frac{|A_m^\delta|}{|\conv(J_m^\delta,A_m^\delta)|}
	\sum_{ n\in\Lambda_m}\frac{ |J_n^\delta|  }{ |
		\conv(J_n^\delta, B_n^\delta)| }
		\lesssim_{\gamma,\beta}
			 \Big( \sum_{j\neq \delta}(1+ |s_j - r_j|)\Big)
			 |r_\delta|^2.
\]
Therefore, coming back to the very beginning of \textsc{Case 2} and inserting this
estimate into the last line of~\eqref{eq:case2}, 
\begin{align*}
	\Big\| \sum_{n\in \Gamma_2} |\langle b,f_n\rangle| \cdot
	|f_n|\Big\|_{L^1(H^c)} &
	\lesssim_{\gamma,\beta}  \sum_{r,s\in\mathbb Z^d} \sum_{\delta=1}^d
	\Big(\sum_{j\neq
	\delta} (1+|s_j - r_j|)\Big) \cdot |r_\delta|^2 q^{|s|_1 + |r|_1} \int_E
	|b(t)|\dif t \\
	&\lesssim \int_E |b(t)|\dif t.
\end{align*} 
Combining now \textsc{Case 1} and \textsc{Case 2} and setting $b=g_j$, we have proved
\eqref{eq:crucialtoshow}.

This completes the proof of our main Theorem~\ref{thm:main}. 

\subsection*{Acknowledgments}
	The author is supported by the Austrian Science Fund FWF, project P32342.

\bibliographystyle{plain}
\bibliography{convergence}

\end{document}